\documentclass[11pt]{amsart}

\usepackage{url}
\usepackage[nocompress]{cite}

\usepackage[english]{babel}
\usepackage[utf8]{inputenc}
\usepackage[T1]{fontenc}

\usepackage{setspace}
\usepackage[letterpaper,left=1in,right=1in,top=1in,bottom=1in]{geometry}

\usepackage{fancyhdr}

\usepackage[bookmarksnumbered]{hyperref}
\hypersetup{
	colorlinks,
}

\usepackage{amsmath,amsthm,amssymb,amsfonts}
\usepackage{mathrsfs} 
\usepackage[shortlabels]{enumitem}

\makeatletter

\def\@tocline#1#2#3#4#5#6#7{\relax
	\ifnum #1>\c@tocdepth 
	\else
	\par \addpenalty\@secpenalty\addvspace{#2}%
	\begingroup \hyphenpenalty\@M
	\@ifempty{#4}{%
		\@tempdima\csname r@tocindent\number#1\endcsname\relax
	}{%
		\@tempdima#4\relax
	}%
	\parindent\z@ \leftskip#3\relax \advance\leftskip\@tempdima\relax
	\rightskip\@pnumwidth plus4em \parfillskip-\@pnumwidth
	#5\leavevmode\hskip-\@tempdima
	\ifcase #1
	\or\or \hskip 1em \or \hskip 2em \else \hskip 3em \fi%
	#6\nobreak\relax
	\dotfill\hbox to\@pnumwidth{\@tocpagenum{#7}}\par
	\nobreak
	\endgroup
	\fi}
\makeatother

\numberwithin{equation}{section}

\newtheorem{Thm}{Theorem}[section]
\newtheorem{Lem}[Thm]{Lemma} 
 
\newtheorem{Prop}[Thm]{Proposition}

\theoremstyle{definition}

\theoremstyle{remark}
\newtheorem{Rmk}{Remark}[section]

\newcommand{\eb}{{\mathrm{e}}}

\newcommand{\PL}{\mathbb{P}_\sigma}

\newcommand{\Norm}[1]{\lVert#1\rVert}
\newcommand{\VNorm}[1]{\left\lVert#1\right\rVert_{V_0}}

\newcommand{\NormXo}[1]{\lVert#1\rVert_{X}}
\newcommand{\NormX}[1]{\lVert#1\rVert_{X}}

\newcommand{\Innerr}[2]{(#1,#2)}

\newcommand{\Abs}[1]{\left|#1\right|}

\newcommand{\loc}{\mathrm{loc}}
\newcommand{\per}{\mathrm{per}}

\newcommand{\dt}{\frac{d}{dt}}

\newcommand{\dtf}[1]{\frac{d#1}{dt}}

\newcommand{\Iht}{\widetilde{I_h}}

\newcommand{\mup}{\mu\nu\lambda_1}
\newcommand{\kt}{k(\nu\lambda_1)^{-1}}
\newcommand{\tunit}{(\nu\lambda_1)^{-1}}
\newcommand{\tdW}{\widetilde{W}}

\newcommand{\ktilde}{\tilde{k}}

\newcommand{\vk}{(w^{(k)},\eta^{(k)})}
\newcommand{\Ik}{\mathcal{I}_k}

\newcommand{\kintv}{[-\kt,T_{**})}

\newcommand{\Tct}{\tilde{\Tc}}

\newcommand{\intave}{\int_{t}^{t+T}}
\newcommand{\intavee}{\int_{t}^{\min(t+T,T_{**})}}

\newcommand{\lam}{\lambda}

\newcommand{\N}{\mathbb{N}}

\newcommand{\R}{\mathbb{R}}

\newcommand{\Tc}{\mathcal{T}}

\newcommand{\As}{\mathscr{A}}

\usepackage{datetime}

\title{A Determining Form for the 2D Rayleigh-B\'enard Problem}

\author{Yu Cao$^{1}$}
\author{Michael S. Jolly$^{1,\dagger}$}
\address{$\dagger$ corresponding author}
\address{$^1$Department of Mathematics\\
	Indiana University\\ Bloomington, IN 47405}

\author{Edriss S. Titi$^2$}
\address{$^2$Department of Mathematics, Texas A\&M University, 3368 TAMU,
	College Station, TX 77843-3368, USA. 
	Department of Computer Science and Applied Mathematics, Weizmann Institute
	of Science, Rehovot 76100, Israel.
	Department of Applied Mathematics and Theoretical Physics, University of Cambridge,  Cambridge CB3 0WA, UK.
} 

\email[Y. Cao]{cao20@iu.edu}
\email[M. S. Jolly]{msjolly@indiana.edu}
\email[E. S. Titi]{Edriss.Titi@damtp.cam.ac.uk and titi@math.tamu.edu}

\date{\today\  at \currenttime}

\begin{document}
	
\begin{abstract}
We construct a determining form for the 2D Rayleigh-B\'enard (RB) system in a strip with solid horizontal boundaries, in the cases of no-slip and stress-free boundary conditions.  The determining form is an ODE in a Banach space of trajectories whose steady states comprise the long-time dynamics of the RB system.  In fact, solutions on the global attractor of the RB system can be further identified through the zeros of a scalar equation to which the ODE reduces for each initial trajectory.  The twist in this work is that the trajectories are for the velocity field only, which in turn determines the corresponding trajectories of the temperature.

\end{abstract}

\maketitle
  \centerline{\it  This paper is dedicated to Ciprian Foias, a great mathematician, generous collaborator and friend,}
  \centerline{\it on the occasion of his 85th birthday.}

\setcounter{tocdepth}{2}
\tableofcontents

\section{Introduction}
It was shown in \cite{foias1987attractors} that the long-time dynamics of the
2D Rayleigh-B\'enard (RB) problem  is entirely contained in the {\it global attractor} $\As$, which is a compact finite-dimensional subset of an infinite-dimensional Hilbert space $H$. 
An inertial manifold, if it exists, is a finite-dimensional invariant smooth manifold that contains the global attractor and attracts all the orbits at an exponential rate (see, e.g., \cite{Foias1988Inertial}). The system obtained by restriction to an inertial manifold is called an \emph{inertial form}. It is a finite-dimensional system of ODEs which reproduces the dynamics of the original system. While the existence of the inertial manifolds has been established for a considerable number of dissipative systems (see, e.g., \cite{temam2012infinite, constantin2012integral,Paret1988Inertial,foiasNicolaenko1988inertial} and references therein), it has been an open problem since the 1980s for the 2D Navier-Stokes equations (NSE), and hence for the 2D RB problem as well.

The 2D NSE and 2D RB problem do enjoy a finite number of {\it determining parameters} (see, e.g., \cite{Foiac-s1967Sur,Jones1993Upper,foias2001navier,Cockburn1995determining}). 
For instance, in the case of determining Fourier modes, if two complete trajectories in the global attractor coincide upon projection $P_m$ on a sufficiently large number $m$ of low Fourier modes, then they must be the same (see, e.g.,  \cite{Foiac-s1967Sur,Jones1993Upper,foias2001navier,Cockburn1995determining}). Thus it is natural to expect the existence of a lifting map $W:P_m \As \to \As$.  This property inspired the notion of a \emph{determining form}, introduced in \cite{foias2012determining}. {A determining form is an
	ODE in an infinite-dimensional Banach space of trajectories that captures the dynamics of the original system in a certain way. Rather than being a dimension reduction, as is the case for the inertial form, the determining form trades the infinite-dimensionality of physical space for that of time; the elements in its phase space are trajectories.  It is an ODE in that it is represented by a globally Lipschitz vector field.}

There are currently two approaches to constructing a determining form.  The key step in either case is to extend the domain of the lifting map $W$ to a Banach space $X$ of projected trajectories.  
The determining form constructed here is based on the nudging approach to continuous data assimilation (see \cite{Azouani2013feedback,azouani2014continuous}).  It is given by
\begin{align}\label{detform2}
	\frac{d v}{ds}=-\|v-I_hW(v)\|^2_X (v-I_hu^*)
\end{align}
where $u^*$ is some steady state of the original system, and $\|\cdot\|_X$ is a sup norm on a Banach space of trajectories that evolve in the finite-dimensional range of some interpolant operator $I_h$.     Note that the evolutionary variable is now $s\in\R$, not time.
The trajectories in the global attractor of the original system are precisely the steady states ($s$-independent solutions) of \eqref{detform2}.  
To show that \eqref{detform2} is an ODE in the true sense boils down to proving that the mapping $W$ is globally
Lipschitz on a ball in $X$, big enough to accomodate $I_h \As$. In addition to the 2D NSE  (see \cite{foias2014unified}), this recipe has been carried out for the damped-driven nonlinear Schr\"odinger, damped-driven Korteweg--de Vries, and surface quasigeostrophic equations (see \cite{Jolly2015determiningNLS, Jolly2017DeterminingKdV, Bai2017DetForm,Jolly2018DeterminingSQG,Biswas2018Down}),  {each with particular treatment and subtle twists in the analysis.} This general procedure is developed in detail in Section \ref{sec3}.

In this paper we construct a determining form for the Rayleigh-B\'enard problem.  The novelty here is that the phase space $X$ corresponds to projections of the velocity field alone. Still, both velocity and temperature of all trajectories in the global attractor of the 2D RB problem are identified through steady states of the determining form.  This is the first such construction where the trajectories are in a subset of the system state variables.  This was suggested in the context of data assimilation by
\cite{Farhat2015continuous,Farhat2017} where it was proved that coarse velocity data alone is sufficient to synchronize with a reference solution of the RB problem. The key difficulty in establishing the crucial Lipschitz property of the lifting map $W$ is in getting \emph{a priori} estimates that are independent of the nudging parameter. Doing this with nudging only in the velocity component adds an extra challenge.

We treat both no-slip and stress-free boundary conditions for the velocity field.   Different analysis is needed for each case.  In the stress-free case, {the problem is equivalent to a periodic boundary condition problem in an extended domain with particular symmetries, which allows us to eliminate one of the nonlinear terms in the estimates.}  On the other hand, we do not in this case have the Poincar\'e inequality for (the first component of ``velocity'') $w$, which is worked around by combining estimates of several norms. We observe that similar techniques are used in \cite{Cao2018Algebraic} to obtain sharper bounds on the size of the global attractor $\As$ in the case of stress-free boundary conditions than previously known.

\section{Notation and Preliminaries}
Under a similar change of variables as in \cite{foias1987attractors}, the 2D RB problem in an infinite strip $\{(x_1,x_2): 0<x_2<l\}$ with solid boundaries at $x_2=0$ and $x_2=l$, can be written as
%
\begin{subequations}\label{eq-boussi0}
	\begin{gather}
		\frac{\partial u}{\partial t}
		-\nu\Delta u+(u\cdot\nabla)u+\nabla p
		=g\theta \eb_2,\\
		\frac{\partial\theta}{\partial t}-\kappa\Delta\theta+(u\cdot\nabla)\theta
		=\frac{u\cdot \eb_2}{l},\\
		\nabla\cdot u=0,\\
		u(0;x)=u_0(x),\quad\theta(0;x)=\theta_0(x),
	\end{gather}
\end{subequations}
where $g$ denotes the gravitational acceleration. Unlike \cite{foias1987attractors}, we retain the dimension of the velocity $u$ while the temperature fluctuation $\theta$ is dimensionless. In this paper, we consider the following two sets of boundary conditions of physical interest.

No-slip:
\begin{align*}
	\textrm{in the $x_2$-variable: }
	&u,\theta=0\ \text{at $x_2=0$ and $x_2=l$,}
	\notag\\
	\textrm{in the $x_1$-variable: }
	&u,\theta,p\ \text{ are of periodic }L.
\end{align*}

Stress-free:
\begin{subequations} 
	\begin{align*}
		\textrm{in the $x_2$-variable: }
		&\frac{\partial u_1}{\partial x_2}, u_2,\theta=0\ \text{at $x_2=0$ and $x_2=l$,}\\
		\textrm{in the $x_1$-variable: }
		&u,\theta,p\ \text{ are of periodic }L.
	\end{align*}
\end{subequations}

\subsection{Function spaces}

We will use the same notation indiscriminately for both scalar and vector Lebesgue and Sobolev spaces, which should not be a source of confusion.

We denote
\begin{gather*}
	(u,v):=\int_{\Omega}u\cdot v\,,
	\ \ |u|:=(u,u)^{1/2}\,,
	\quad \text{for } u,v\in L^2(\Omega)\,,
	\\
	((u,v)):=\int_{\Omega}\nabla u\cdot\nabla v\,,
	\ \ \Norm{u}:=((u,u))^{1/2}\,,
	\quad \text{for } \nabla u,\nabla v\in L^2(\Omega)\,,
\end{gather*}
for a domain $\Omega$ that will be specified for each case of boundary conditions.

\subsubsection{No-slip BCs}

We define function spaces corresponding to the no-slip boundary conditions as in \cite{Farhat2015continuous}. Let $\Omega=\Omega_0:=(0,L)\times(0,l)$ and $\mathcal{F}$ be the set of $C^\infty(\Omega)$ functions, which are trigonometric polynomials in $x_1$ with period $L$, and compactly supported in the $x_2$-direction. 

Denote the space of smooth vector-valued functions which incorporates the divergence-free condition by
\begin{align*}
	\mathcal{V}:=\{u\in\mathcal{F}\times\mathcal{F}:
	\nabla\cdot u=0\}\,,
\end{align*}
and the closures of $\mathcal{V}$ and $\mathcal{F}$ in $L^2(\Omega)$ by $H_0$ and $H_1$, respectively, which are endowed with the usual inner products
and associated norms
\begin{align}\label{Hspace}
	(u,v)_{H_0}:=(u,v)\,,\quad
	(\psi,\phi)_{H_1}:=(\psi,\phi)\,,\quad
	\Norm{u}_{H_0}:=(u,u)^{1/2}\,,\quad 
	\Norm{\psi}_{H_1}:=(\psi,\psi)^{1/2}\,.
\end{align}
The closures of $\mathcal{V}$ and $\mathcal{F}$ in $H^1(\Omega)$ will be denoted by $V_0$ and $V_1$, respectively, endowed with the inner products and associated norms
\[
((u,v))_{V_0}:=
((u,v))\,,\quad
((\psi,\phi))_{V_1}:=((\psi,\phi))\,,
\quad
\Norm{u}_{V_0}:=\Norm{u}\,,\quad
\Norm{\phi}_{V_1}:=\Norm{\phi}\,.
\]

\subsubsection{Stress-free BCs}
Following \cite{Farhat2017}, we consider the equivalent formulation of the 2D RB problem \eqref{eq-boussi0} subject to the fully periodic boundary conditions on the extended domain $\Omega=(0,L)\times(-l,l)$ with the following special spatial symmetries: for $(x_1,x_2)\in\Omega$,
\begin{align*}
	u_1(x_1,x_2)&=u_1(x_1,-x_2)\,,\quad u_2(x_1,x_2)=-u_2(x_1,-x_2)\,,\\
	p(x_1,x_2)&=p(x_1,-x_2)\,,\quad\quad\theta(x_1,x_2)=-\theta(x_1,-x_2)\,.
\end{align*}
{Observe that for $(x_1,x_2)\in\Omega$ with $x_2=-l,0,l$, and for smooth enough functions one has
	\begin{align*}
		\frac{\partial u_1}{\partial x_2}, u_2,\theta=0\,,
	\end{align*}
	that is, one recovers the original corresponding physical boundary conditions when restricted to the physical domain $\Omega_0$.}

We define function spaces corresponding to the ``stress-free'' boundary conditions, i.e., the periodic BCs with the above symmetries, as in \cite{Farhat2017}, where
\begin{quote}
	$\mathcal{F}_1$ is the set of trigonometric polynomials in $(x_1,x_2)$, with period $L$ in the $x_1$-variable, that are even, with period $2l$, in the $x_2$-variable,
\end{quote}
and
\begin{quote}
	$\mathcal{F}_2$ is the set of trigonometric polynomials in $(x_1,x_2)$, with period $L$ in the $x_1$-variable, that are odd, with period $2l$, in the $x_2$-variable.
\end{quote}
The symmetries of the two velocity components lead us to take in the stress-free case
\begin{align*}
	\mathcal{V}:=\{u\in\mathcal{F}_1\times\mathcal{F}_2:
	\nabla\cdot u=0\}\,.
\end{align*}
The space $H_0$ will again be the closure of $\mathcal{V}$ in $L^2(\Omega)$, but $H_1$ shall be that of $\mathcal{F}_2$ in $L^2(\Omega)$, with inner products and norms as in \eqref{Hspace}.

Similarly, we denote the closures of $\mathcal{V}$ and $\mathcal{F}_2$ in $H_{\per}^1(\Omega)$ by $V_0$ and $V_1$, respectively, but with the inner products
\[
((u,v))_{V_0}:=\frac{1}{|\Omega|}(u,v)+((u,v))\,,
\quad
((\psi,\phi))_{V_1}:=((\psi,\phi))\,,
\]
and associated norms
\[
\Norm{u}_{V_0}:=\left(\frac{1}{|\Omega|}|u|^2+\Norm{u}^2\right)^{1/2}\,,
\quad
\Norm{\phi}_{V_1}:=\Norm{\phi}\,.
\]

\subsection{The linear operators $A_i$}\label{opA}
\subsubsection{No-slip BCs}
Let $A_i:D(A_i)\to H_i$ ($i=0,1$) be the unbounded linear operators defined by
\begin{align*}
	(A_iu,v)_{H_i}=((u,v))_{V_i},\quad i=0,1,\quad \forall\ u,v\in D(A_i)\;,
\end{align*}
where  $D(A_0)=V_0\cap H^2(\Omega)$ and $D(A_1)=V_1\cap H^2(\Omega)$. 

For each $i=0,1$, the operator $A_i$ is self-adjoint and $A_i^{-1}$ is a compact, positive-definite, self-adjoint linear operator in $H_i$. There exists a complete orthonormal set of eigenfunctions $(\zeta_{i,j})_{j=1}^\infty$ in $H_i$ such that $A_i\zeta_{i,j}=\lambda_{i,j}\zeta_{i,j}$ where
\[
0<\lambda_{i,1}\leqslant\lambda_{i,2}
\leqslant\cdots\leqslant \lambda_{i,m}\leqslant
\cdots,
\]
Observe that we have the following Poincar\'{e} inequalities:
\begin{align}
	|\phi|^2
	\leqslant \lambda_1^{-1}\Norm{\phi}^2,\quad
	&\forall \,  \phi \in V_i,\\
	\Norm{\phi}^2
	\leqslant \lambda_1^{-1}|A_1\phi|^2,\quad
	&\forall \,  \phi\in D(A_i),
\end{align}
where $\lambda_1:=\lambda_{1,1}=\lambda_{2,1}$. 

\begin{Rmk}
	We observe that in this case $|A_0\phi|$ is equivalent to $\Norm{\phi}_{H^2}$ for every $\phi\in D(A_0)$.
\end{Rmk}

\subsubsection{Stress-free BCs}
Let $A_i:D(A_i)\to H_i$ ($i=0,1$) be the unbounded linear operators defined by $A_i=-\Delta$, where $D(A_0)=V_0\cap H^2(\Omega)$ and $D(A_1)=V_1\cap H^2(\Omega)$. 



\begin{Rmk}
	The operator $A_0$ is a nonnegative operator and possesses a sequence of eigenvalues with
	\[
	0=\lambda_{0,1}<\lambda_{0,2}
	\leqslant\cdots\leqslant \lambda_{0,m}\leqslant
	\cdots,
	\]
	associated with an orthonormal basis $\{\zeta_{0,m}\}_{m\in\N}$ of $H_0$. The operator $A_1$ is a positive self-adjoint operator and possesses a sequence of eigenvalues with 
	\[
	0<\lambda_{1,1}\leqslant\lambda_{1,2}
	\leqslant\cdots\leqslant \lambda_{1,m}\leqslant
	\cdots,
	\]
	associated with an orthonormal basis $\{\zeta_{1,m}\}_{m\in\N}$ of $H_1$. Observe that we have the Poincar\'{e} inequality for temperature:
	\begin{align}
		|\theta|^2
		\leqslant \lambda_1^{-1}\Norm{\theta}^2,\quad
		&\forall \,  \theta \in V_1,\\
		\Norm{\theta}^2
		\leqslant \lambda_1^{-1}|A_1\theta|^2,\quad
		&\forall \,  \theta\in D(A_1),
	\end{align}
	where $\lambda_1=\lambda_{1,1}$. 
\end{Rmk}
\begin{Rmk}
	In the stress-free case, we do not have the Poincar\'{e} inequality for functions in $V_0$, but we have
	\begin{align}\label{u-Poin}
		|u|^2\leqslant |\Omega|\VNorm{u}^2,\quad\forall \, u\in V_0
	\end{align}
	by the definition of the norm $\VNorm{\cdot}$.
\end{Rmk}

\begin{Rmk}\label{rmkreg}
	By the elliptic regularity of the operator $A_0+I$ (see \cite[Remark 2.3]{Farhat2017}), we have in the stress-free case the equivalency
	\begin{align}\label{elliptic-reg}
		\widetilde{c_E}^2
		\bigg(
		\frac{1}{|\Omega|}\Norm{u}_{L^2}+\Norm{A_0u}_{L^2}
		\bigg)\leqslant
		\Norm{u}_{H^2}\leqslant
		c_E^2
		\bigg(
		\frac{1}{|\Omega|}\Norm{u}_{L^2}+\Norm{A_0u}_{L^2}
		\bigg),
		\quad \forall\ u\in D(A_0).
	\end{align}
\end{Rmk}

\subsection{The bilinear maps $B_i$}
Denote the dual space of $V_i$ by $V_i'$ ($i=0,1$). Define the bilinear map $B_0: V_0\times V_0\to V_0'$ (and the trilinear map $b_0: V_0\times V_0\times V_0'\to\R$) by the continuous extension of 
\[
b_0(u,v,w):=\langle B_0(u,v),w\rangle_{V_0'}=((u\cdot\nabla)v,w), \quad
u,v,w\in\mathcal{V}.
\]
\subsubsection{No-slip BCs}
Define the scalar analogue $B_1:V_0\times V_1\to V_1'$ (and the trilinear map $b_1: V_0\times V_1\times V_1'\to\R$) by the continuous extension of 
\[
b_1(u,\theta,\phi):=\langle B_1(u,\theta),\phi\rangle_{V_1'}=((u\cdot\nabla)\theta,\phi), \quad
u\in\mathcal{V},\ \ \theta,\phi\in\mathcal{F}.
\]


The bilinear maps $B_i$ (and the trilinear maps $b_i$), $i=0,1$, have the orthogonality property: 
\begin{align}\label{ns-orth}
	b_0(u,v,v)=0,\quad
	b_1(u,\theta,\theta)=0,\quad u,v\in V_0,\  \theta\in V_1.
\end{align}

\subsubsection{Stress-free BCs}
Define the scalar analogue $B_1:V_0\times V_1\to V_1'$ (and the trilinear map $b_1: V_0\times V_1\times V_1'\to\R$) by the continuous extension of 
\[
b_1(u,\theta,\phi):=\langle B_1(u,\theta),\phi\rangle_{V_1'}=((u\cdot\nabla)\theta,\phi), \quad
u\in\mathcal{V},\ \ \theta,\phi\in\mathcal{F}_2.
\]

The bilinear maps $B_i$ (and the trilinear maps $b_i$), $i=0,1$, have the same orthogonality property \eqref{ns-orth} as in the no-slip case. Furthermore, we have for each $u\in D(A_0)$, 
\begin{align}
	b_0(u,u,A_0u)=0\,,
\end{align}
which is not true in general in the no-slip case.

\subsection{Functional setting and bounds for the global attractor}
Following \cite{foias1987attractors}, we have the functional form of the RB problem (\ref{eq-boussi0}):
\begin{subequations}\label{eq-benardfn0}
	\begin{gather}
		\label{eq:1}
		\frac{du}{dt}+\nu A_0u+B_0(u,u)=\PL(g\theta \eb_2),\\
		\label{eq:2}
		\frac{d\theta}{dt}+\kappa A_1\theta+B_1(u,\theta)=\frac{u\cdot \eb_2}{l},\\
		\label{eq:3}
		u(0;x)=u_0(x),\quad\theta(0;x)=\theta_0(x),
	\end{gather}
\end{subequations}
where $\PL$ denotes the Helmholtz-Leray projector from $L^2(\Omega)$ onto $H_0$.

\subsubsection{No-slip BCs}
It is shown in \cite{foias1987attractors} that the RB system \eqref{eq-boussi0} with no-slip boundary conditions has a global attractor
\begin{align}\label{attr}
	\As=\{(u_0,\theta_0)\in H_0\times H_1: \exists\,\textrm{a unique solution } (u,\theta)(t;u_0,\theta_0) \textrm{ of \eqref{eq-boussi0} for all $t\in\R$ }\\ \notag \textrm{and }\sup_t(\VNorm{u(t)}+\Norm{\theta(t)}_{V_1})<\infty \}\,.
\end{align}
Alternatively, $\As$ is the maximal bounded invariant subset of $V_0\times V_1$ under the dynamics of (\ref{eq-benardfn0}).   
Moreover, there exists some (dimensional) constants $J_i>0$, $i=1,2$, such that  
\begin{align}\label{attractor-bnd-ns}
	\sup_{t\in\R}\VNorm{u(t)}\leqslant J_1, 
	\quad 
	\sup_{t\in\R}\Norm{u(t)}_{H^2}\leqslant J_2,
	\quad\forall \, (u,\theta)\in\As.
\end{align}
Henceforth, lowercase letters $c_L, c_A,  c_i, \cdots$ will denote universal dimensionless positive constants; uppercase letters $C,J_i,K,K_i,\cdots$ will denote positive dimensional constants that {depend on the physical parameters.}

\subsubsection{Stress-free BCs}
The case of stress-free boundary conditions is studied further in \cite{Cao2018Algebraic}. With the stress-free boundary conditions, the RB system has steady states with arbitrarily large $L^2$-norms:
\[
u(x)=(c,0),\quad \theta(x)=0,\quad c\in\R,
\]
which means that the system is not dissipative. {However, since (see also \cite{Cao2018Algebraic}) 
	\[
	\dt\int_{\Omega}u(x,t)\ dx=0\,,
	\]
	we may assume in the stress-free case that the velocity field has a fixed average:
	\begin{align}\label{space_average}
		\int_{\Omega}u(x,t)\ dx=a,\quad\forall\ t\in\R,
	\end{align}
	where $a\in \R$ is fixed. Observe that the spatial average is conserved and the system is dissipative within each invariant affine space of fixed average $a$.} It is shown in \cite{Cao2018Algebraic} that the RB system has a global attractor $\As=\As_a$, in each affine subspace of $V_0\times V_1$ where the spatial average \eqref{space_average} of velocity is fixed.  Moreover, there exist some (dimensional) constants $J_i=J_i(a)>0$, $i=1,2$, such that \eqref{attractor-bnd-ns} holds.
In this case of stress-free boundary conditions, the dependence of $J_i$, $i=1,2$, is shown in \cite{Cao2018Algebraic} to be algebraic in the physical parameters $\nu$, $\kappa$, $l$ and $L$. To be specific, we will take $a=0$.

\section{Determining Form and Main Results}\label{sec3}
In order to define the determining form, we need the notion of interpolant operators.
\subsection{Interpolant operators}
We recall  a general class of interpolant operators introduced in \cite{azouani2014continuous,Azouani2013feedback} for dealing
with various determining parameters such as modes, nodes, volume elements, etc. These operators are finite-rank operators (bounded, linear and with finite-dimensional range) and are required to
satisfy an approximation of identity type condition.

A finite-rank operator $I_h:H_{}^1(\Omega)\to H_{}^1(\Omega)$ is a \emph{Type I interpolant operator} if it satisfies
\begin{gather}\label{intplt-t1}
	\Abs{\varphi-I_h(\varphi)}
	\leqslant c_0 h\|\varphi\|_{H_{}^1},\quad\forall\, \varphi\in H_{}^1\,;\\
	\Norm{\varphi-I_h(\varphi)}_{H^1}\leqslant
	\tilde{c_0}\Norm{\varphi}_{H_{}^1},\quad \forall\, \varphi\in H_{}^1\,.
\end{gather}
A finite-rank operator $I_h:H^2(\Omega)\to H^1(\Omega)$ is a \emph{Type II interpolant operator} if it satisfies
\begin{gather}\label{intplt-t2}
	|\varphi-I_h(\varphi)|
	\leqslant 
	c_1h\Norm{\varphi}_{H_{}^1}
	+
	c_2h^2\Norm{\varphi}_{H_{}^2},\quad\forall\, \varphi\in  H_{}^2\;;\\
	\label{int-t2}
	\Norm{\varphi-I_h(\varphi)}_{H^1}\leqslant
	\tilde{c_1}\Norm{\varphi}_{H_{}^1}
	+
	\tilde{c_2}h\Norm{\varphi}_{H_{}^2},\quad\forall\, \varphi\in  H_{}^2.
\end{gather}


In this paper, we construct a determining form for the RB system using Type II interpolants. The same can be done under slightly weaker assumptions on $h$ for Type I interpolants (see \cite{CaoThesis}). 

\begin{Rmk}
	The orthogonal projection onto low Fourier modes, those with
	wave numbers $k$ such that $|k| \leqslant 1/h$, is one example of a Type I interpolant.
	Another is finite volume elements.  In addition, an example of a Type II interpolant is an interpolant operator that is based on nodal values satisfying (\ref{intplt-t2}) and \eqref{int-t2}. See, e.g.,  \cite{azouani2014continuous} for more details.
\end{Rmk}

\begin{Rmk}\label{rmksf}
	In the stress-free case, by definition, we have  $\Norm{\varphi}_{H^1}=\VNorm{\varphi}$, for $\varphi\in V_0$. Moreover, by \eqref{elliptic-reg} in Remark \ref{rmkreg}, replacing the absolute constants when necessary, we can replace $\Norm{\varphi}_{H^2}$ by $|A_0\varphi|$ in \eqref{intplt-t2} and \eqref{int-t2}, for $\varphi\in D(A_0)$.
\end{Rmk}

We need to modify the interpolant operator $I_h$ so that its has a range of functions that are divergence-free and satisfy the boundary conditions. Motivated by \cite[Proposition 2.1]{Celik2018Spectral}, we define the modified Type II interpolant operator
$\Iht: H^2\to V_0$ as
\begin{align}\label{E1}
	\Iht:=P_r I_h,\quad 
	P_r\phi=\sum_{i=1}^r (\phi,\zeta_{0,i})\zeta_{0,i},
	\quad 
	h^2\sim \frac{1}{\lambda_{0,r}},
\end{align}
where we recall that $\{\zeta_{0,i}\}$ are the eigenfunctions of the operator $A_0$ in Section \ref{opA}.
The phase space $(X,\Norm{\cdot}_X)$ of our determining form is then defined as
\begin{align}\label{Xnorm}
	X:=C_b(\R;\Iht H^2),\quad \Norm{v}_X:=\frac{\sup_{t\in\R}\VNorm{v(t)}}{\nu\lambda_1^{1/2}}\, .
\end{align}

\begin{Rmk}\label{Iht}
	Based on the proof in \cite[Proposition 2.1]{Celik2018Spectral}, we observe that $\Iht$ satisfies conditions \eqref{intplt-t2} and \eqref{int-t2} with modified constants $c_i, \tilde{c_i}$, $i=1,2$. Furthermore, in the no-slip case, by the Poincar\'e inequality, modifying the constants $c_i$ when necessary, we have
	\begin{align}\label{mint}
		|\varphi-\Iht(\varphi)|
		\leqslant 
		c_1 h\Norm{\varphi}_{V_0}
		+
		c_2 h^2\Abs{A_0\varphi},\quad\forall\, \varphi\in  D(A_0).
	\end{align}
	We also have \eqref{mint} for the stress-free case by Remark \ref{rmksf}.
\end{Rmk}

\subsection{Auxiliary system and  determining map}
Consider the following auxiliary system:
\begin{subequations}\label{eq-aux0}
	\begin{align}
		\frac{dw}{dt}+\nu A_0w+B_0(w,w)
		&=\PL(g\eta \eb_2)-\mup (\Iht w-v),\label{eq-aux}\\
		\frac{d\eta}{dt}+\kappa A_1\eta+B_1(w,\eta)
		&=\frac{w\cdot \eb_2}{l},\label{eq-aux2}
	\end{align}
\end{subequations}
where $v\in B_X(0,\rho):=\{\xi\in X: \Norm{\xi}_X<\rho \}$ with $\rho>0$ and $\Iht$ is a (modified) Type II interpolant operator.  Note that the nudging term in \eqref{eq-aux0} appears only in the momentum equation.

\begin{Prop}[Solutions to the auxiliary system]
	\label{Wmap-Exst-t1}
	Let $\rho$ be a positive real number. Let $\mu>0$ be sufficiently large and $h>0$ sufficiently small (see conditions in Section \ref{SecPropAuxNoSlip}). 
	Then for each $v\in B_X(0,\rho)$, system (\ref{eq-aux0}) has a unique bounded solution $(w(t),\eta(t))$ that exists for all $t\in\R$ such that
	\begin{align}\label{aux-soln}
		(w,\eta)\in C_b(\R,V_0\times V_1)\cap L_{\loc}^2(\R,D(A_0)\times D(A_1)),\quad
		\bigg(\dtf{w},\dtf{\eta}\bigg)\in L_{\loc}^2(\R,H_0\times H_1)\,.
	\end{align}
\end{Prop}

The proof of Proposition \ref{Wmap-Exst-t1} is given in Section \ref{SecPropAuxNoSlip}. Note that this proposition provides a map, called the \emph{determining map},
\[
\widetilde{W}:B_X(0,\rho)\to C_b\big(\R;V_0\times V_1\big)\cap L^2_{\loc}\big(\R;D(A_0)\times D(A_1)\big),\quad \widetilde{W}(v):=(w,\eta).
\]
The projection of $\widetilde W$ to the first component $w$ induces a map $W: B_X(0,\rho)\to Y$ with
\begin{gather*}
	Y:=C_b(\R;V_0)\cap L^2_{\loc}(\R;D(A_0)),\quad W(v)=w\,,\\  \|w\|_Y:=\frac{\sup_{t\in\R}\VNorm{w(t)}}{\nu\lambda_1^{1/2}} + 
	\left(
	\frac{1}{\nu\lambda_1}\sup_{t\in\R} \int_{t}^{t+\frac{1}{\nu\lambda_1}}|A_0w(\tau)|^2\,d\tau
	\right)^{1/2}\,.
\end{gather*}
The induced map $W$ will be used in the definition of the determining form. We denote $Z:=C_b(\R;V_1)\cap L^2_{\loc}(\R;D(A_1))$ and
\begin{align*}
	\|\eta\|_Z:={\sup_{t\in\R}\Norm{\eta(t)}_{V_1}} + 
	\left(
	\nu\sup_{t\in\R} \int_{t}^{t+\frac{1}{\nu\lambda_1}}|A_1\eta(\tau)|^2\,d\tau
	\right)^{1/2}\,.
\end{align*}
\begin{Prop}\label{cor-Lip}
	The maps $\tdW:(B_X(0,\rho), \Norm{\cdot}_X)\to (Y\times Z,\Norm{\cdot}_Y+\Norm{\cdot}_Z)$ and $W:B_X(0,\rho),\Norm{\cdot}_X\to(Y,\Norm{\cdot}_Y)$
	are Lipschitz.
\end{Prop}
The proof of Proposition \ref{cor-Lip} is given in Section \ref{varphibnd}.

\begin{Rmk}
	It is proved in \cite{Biswas2018Down}  that the determining map $\widetilde{W}$ is in fact Frech\'{e}t differentiable in the case of the 2D NSE. 
\end{Rmk}

\subsection{Determining form and long-time dynamics of the RB system}\label{sec:Lip-F}
Let  $(u^*, \theta^*)$ be a steady state of the RB problem (\ref{eq-benardfn0}); for instance, we may take $(u^*, \theta^*)=(0,0)$.
Under the assumptions of Proposition \ref{Wmap-Exst-t1}, we will prove (in Theorem \ref{thm-detForm} (i)) that the differential equation
\begin{gather}\label{eq-DetForm}
	\frac{dv}{ds}=F(v):=-\NormXo{v-\Iht W(v)}^2(v-\Iht u^*),\quad v(0)=v_0\in B_X(0,\rho),
\end{gather}
is an \emph{ODE} in the sense that the vector field $F$ is globally Lipschitz in the ball $B_X(0,\rho)$, where $\rho>0$ is to be determined. 
The ODE \eqref{eq-DetForm} is called a \emph{determining form} of the RB problem.


The connection between the long-time dynamics, i.e. the global attractor, of the RB problem (\ref{eq-benardfn0}) and the determining form will be made through the following result:
\begin{Prop}\label{prop-detForm-RB}
	Let $(u(t),\theta(t))$, $t\in\R$, be a solution of the RB problem (\ref{eq-benardfn0}) that lies in the global attractor $\As$. Suppose $\mu, h$ satisfy the assumptions in 
	Proposition \ref{Wmap-Exst-t1}.
	Suppose $(w,\eta)$ is a solution to the system 
	\begin{subequations}\label{detform-RB}
		\begin{align}
			\label{eq1:detformRB}
			\frac{dw}{dt}+\nu A_0w+B_0(w,w)
			&=\PL(g\eta \eb_2)-\mup (\Iht w-\Iht u),\\
			\label{eq2:detformRB}
			\frac{d\eta}{dt}+\kappa A_1\eta+B_1(w,\eta)
			&=\frac{w\cdot \eb_2}{l},
		\end{align}
	\end{subequations}
	and satifies
	\begin{align*}
		(w,\eta)\in C_b(\R,V_0\times V_1)\cap L_{\loc}^2(\R,D(A_0)\times D(A_1))\,,
		\quad
		\bigg(\dtf{w},\dtf{\eta}\bigg)\in L_{\loc}^2(\R,H_0\times H_1)\,.
	\end{align*}
	Then $(w(t),\eta(t))=(u(t),\theta(t))$ for all $t\in\R$.
\end{Prop}

The proof of Proposition \ref{prop-detForm-RB} is given in Section \ref{pf2}.

\subsection{Main theorem}
In order to state the main theorem, we first prove the following result:
\begin{Prop}\label{prop-Ju-bnd}
	Let $\Iht$ be a (modified) Type II interpolant operator as in \eqref{E1}, with $h<L$. For every $(u,\theta)\in\mathscr{A}$,
	we have
	\begin{align}\label{Ju-bnd}
		\Norm{\Iht u}_X\leqslant R:
		=\big((\tilde{c_1}+1)J_1+\tilde{c_2}L J_2\big)/(\nu\lambda_1^{1/2}).
	\end{align}
\end{Prop}

\begin{proof}
	Let $(u,\theta)\in\mathscr{A}$. 
	By \eqref{int-t2}, Remark \ref{Iht}, and the bound \eqref{attractor-bnd-ns}, we have
	\begin{align*}
		\Norm{\Iht u}_{V_0}
		&\leqslant
		\Norm{\Iht u-u}_{V_0}+\VNorm{u}\\ \notag
		&\leqslant
		\tilde{c_1}\Norm{u}_{H^1}
		+\tilde{c_2}h\Norm{u}_{H^2}
		+\VNorm{u}
		\leqslant
		(\tilde{c_1}+1)J_1+\tilde{c_2}L J_2\,,
	\end{align*}
	which completes the proof by \eqref{Xnorm}, the definition of the norm $\NormX{\cdot}$.
\end{proof}

The main results regarding the determining form are summarized in the following theorem:
\begin{Thm}\label{thm-detForm}
	Suppose the assumptions in Proposition \ref{Wmap-Exst-t1} hold for $\rho=4R$, where $R>0$ satisfies (\ref{Ju-bnd}). Suppose also that $h<L$ as in Proposition \ref{prop-Ju-bnd}.
	Then the  following hold.
	\begin{enumerate}[(i)]
		\item The vector field $F: B_X(0,\rho)\to X$ in the determining form (\ref{eq-DetForm}) is Lipschitz. Hence the determining form (\ref{eq-DetForm}) is an ODE in $X$ which has short-time existence and uniqueness of solutions for every initial data $v_0\in B_X(0,\rho)$.
		\item The ball $B_X(\Iht u^*, 3R)\subset B_X(0,\rho)$ is  forward invariant in the evolution variable $s$ under the dynamics of the determining form, which implies that (\ref{eq-DetForm}) has a unique global solution for every initial data $v_0\in B_X(\Iht u^*,3R)$.
		\item Every solution of (\ref{eq-DetForm}) with initial data $v_0\in B_X(\Iht u^*,3R)$ converges to a steady state of (\ref{eq-DetForm}) as $s\to\infty$.
		\item All the steady states of the determining form (\ref{eq-DetForm}) that are contained in $B_X(0,\rho)$ have the form $v(t)=\Iht u(t)$ for all $t\in\R$, where $(u(\cdot),\theta(\cdot))$ is a trajectory in the global attractor $\As$ of the RB problem (\ref{eq-benardfn0}) for a uniquely determined termperature $\theta(\cdot)$.
	\end{enumerate} 
\end{Thm}

We should emphasize that \eqref{eq-DetForm} governs an evolution of ``trajectories'' that are with range in a finite-dimensional space which correspond to velocity only.  Yet it determines full trajectories of both the velocity and temperature on the global attractor of the RB system through the determining map $\tdW$.

\begin{Rmk}
	It is easy to see, as in \cite{Foias2017One}, that the solution to \eqref{eq-DetForm} is always 
	a convex combination of the initial condition and the chosen steady state:
	\begin{align} \label{combo}
		v(s;t) = \beta(s) v_0(t) + (1-\beta(s)) \Iht u^*\quad s\geqslant 0,\  t\in\R
		\;,
	\end{align}
	where 
	\begin{align}\label{beta}
		\beta (s) =  \exp \left( - \int _0 ^ s \| v(\tau) - \Iht W(v(\tau)) \| _{X} ^ 2 \, d\tau \right)
	\end{align}
	satisfies
	a scalar ODE, which for 
	{the RB problem written in the form \eqref{eq-benardfn0}
		with $(u^*,\theta^*)=(0,0)$, amounts to
		\begin{equation}\label{one-p}
		v = \beta v_0 \;, \qquad \frac{d\beta}{ds} =  -\beta \|  \beta v_0 - \Iht W(\beta v_0) \| _{X} ^2 \;,\quad  \beta(0) = 1.
		\end{equation}
		The dynamics of \eqref{one-p} are completely understood (see \cite{Foias2017One}). 
		As $s\to \infty$, along the straight line through $v_0$ and $0$ in $X$, either $v(s) \to 0$, or $v(s) \to \Iht u$, where $( u, \theta)$ is the first trajectory in $\As$, with $\Iht u$ between $v_0$ and $0$.} Thus the solutions in the global attractor can be identified as the zeros of the scalar function on the right-hand side of equation \eqref{one-p}. 
\end{Rmk}

\begin{proof} [Proof of Theorem \ref{thm-detForm}]
	
	Part (i). Define $q:B_X(0,\rho)\to\R$ with $q(v):=\NormXo{v-\Iht W(v)}$. Let $v_1,v_2\in B_X(0,\rho)$. 
	By the triangle inequality and the definition of the vector field $F$, 
	\begin{align*}
		\NormX{F(v_1)-F(v_2)}
		=
		\NormX{[q^2(v_1)-q^2(v_2)](v_1-\Iht u^*)+q^2(v_2)(v_1-v_2)}\\
		\leqslant
		\Abs{q^2(v_1)-q^2(v_2)}\cdot\NormX{v_1-\Iht u^*}
		+\Abs{q^2(v_2)}\cdot\NormX{v_1-v_2}.
	\end{align*}
	Hence, to show that $F$ is Lipschitz (in the ball $B_X(0,\rho)$), it suffices to show that the map $q$ is Lipschitz. Note that
	\begin{align*}
		\Abs{q(v_1)-q(v_2)}
		&=\big|\NormXo{v_1-\Iht W(v_1)}-\NormXo{v_2-\Iht W(v_2)}\big|\\ \notag
		&\leqslant
		\NormXo{{v_1-\Iht W(v_1)}-[v_2-\Iht W(v_2)]} \\ \notag
		&\leqslant
		\NormXo{v_1-v_2}+\NormXo{\Iht W(v_1)-\Iht W(v_2)}.
	\end{align*}
	It suffices to show that
	\begin{align}\label{Lip-ineq}
		\NormXo{\Iht W(v_1)-\Iht W(v_2)}
		\leqslant c \NormX{v_1-v_2}.
	\end{align}
	Observe the following diagram:
	\begin{align*}
		B_X(0,\rho)\subset(X,\NormX{\cdot})
		\xrightarrow{\textrm{$W$}}
		(Y,\Norm{\cdot}_Y)
		\xrightarrow{\textrm{$\Iht$}}
		(X,\NormXo{\cdot}).
	\end{align*}
	To prove \eqref{Lip-ineq}, it suffices to show that
	\begin{align}\label{lip1}
		\Norm{w_1-w_2}_Y
		&\leqslant c
		\NormX{v_1-v_2},\\\label{lip2}
		\NormXo{\Iht w_1-\Iht w_2}
		&\leqslant c
		\Norm{w_1-w_2}_Y,
	\end{align}
	where $w_i:=W(v_i)$ with $i=1,2$.
	
	
	Proposition \ref{cor-Lip} implies that $W$ is Lipschitz and hence we have \eqref{lip1}. Inequality \eqref{lip2} follows from Remark \ref{Iht} for the linear operator $\Iht$ and the definitions of the norms $\NormXo{\cdot}$ and $\Norm{\cdot}_Y$. The proof of (i) is done. 
	
	By Proposition \ref{prop-Ju-bnd}
	and the triangle inequality\footnote{
		Note that $\Norm{v}_X\leqslant \Norm{v-\Iht u^*}_X+\Norm{\Iht u^*}_X\leqslant 3R+R=4R.$}, 
	\[
	B_X(\Iht (u^*), 3R)\subset B_X(0,\rho),
	\]
	which implies short-time existence  of a solution of the determining form \eqref{eq-DetForm}. Thus, (ii) follows from the observation that
	\begin{align*}
		\NormX{v(s;\cdot)-\Iht(u^{*})}=\beta(s)\NormX{v_0(\cdot)-\Iht (u^*)},\quad s\geqslant 0,
	\end{align*}
	where $\beta$ is as in \eqref{beta}. Alternatively, (ii) follows from the dissipativity property of \eqref{eq-DetForm}: for every fixed $t\in\R$,
	\begin{align*}
		\frac{d}{ds}\Norm{v(s;t)-\Iht(u^*)}^2_{V_0}
		=
		-2\NormXo{v-\Iht W(v)}^2\cdot
		\Norm{v(s;t)-\Iht(u^*)}^2_{V_0}.
	\end{align*}
	This property implies that the ball $B_X(\Iht (u^*), 3R)$ is forward invariant for all $s\geqslant 0$, which proves both  (ii) and  (iii).
	
	To prove  (iv) we observe that the steady states of equation \eqref{eq-DetForm} in the ball $B_X(0,\rho)$ are either $v=\Iht (u^*)$ or $v\in B_X(0,\rho)$ such that $\NormX{v-\Iht W(v)}=0$.
	In the first case $(u^*, \theta^*)\in\mathscr{A}$ since $(u^*, \theta^*)$ is a steady state of the RB system \eqref{eq-benardfn0}. 
	In the second case we have $v(t)=\Iht W(v)(t)$ for all $t\in\R$. Let $(w,\eta)=\widetilde{W}(v)$. It then follows from \eqref{eq-aux0} that $(w, \eta)$ is  a bounded solution (thus a trajectory in the global attractor $\mathscr{A}$ by \eqref{attr}) to the RB system (\ref{eq-benardfn0}). 
	
	Conversely, since $\rho=4R$, it follows from Proposition \ref{prop-Ju-bnd} that 
	\[
	\Iht (\mathscr{A})\subset B_X(\Iht u^*, 3R)\subset B_X(0,\rho).
	\] 
	Thus, for every trajectory $(u(\cdot),\theta(\cdot))\subset\mathscr{A}$ it follows from the auxiliary system (\ref{eq-aux0}) and Proposition \ref{prop-detForm-RB} that $u(t)=W(\Iht u)(t)$ for all $t\in\R$. In particular, $\Iht u=\Iht W(\Iht u)$, which implies that $\Iht u$ is a steady state of equation (\ref{eq-DetForm}) in $B_X(0,\rho)$.
\end{proof}

\section{Proof of Proposition \ref{Wmap-Exst-t1}}\label{SecPropAuxNoSlip}
Let $\mu,h>0$ and assume that $\NormX{v}\leqslant \rho$. For the case of no-slip boundary conditions, we assume that the following hold:
\begin{align}\label{muh}
	\mu\lam_1^{1/2}c_1h\leqslant\frac{1}{4}\,,\quad
	\mu\lam_1^22c_2^2h^4\leqslant \frac{1}{8}\,,
\end{align}
\begin{gather}\label{mu5}
	\mu\nu^2\lam_1^2 C_1>\frac{5g^2K}{2\rho^2}\,,
\end{gather}
\begin{gather}\label{mu3}
	\frac14\mu\nu - 16K_1C_1^2\rho^4>0\, ,
\end{gather}
\begin{align}\label{mu4}
	\frac12\mup
	-\frac{g^2}{\kappa(\nu\lambda_1)^2}
	-\frac{\lambda_1\nu}{4}(K_2\log K_2)
	-\frac{2c_L^2\nu^2}{\kappa}\rho^2
	-\frac{2\nu^2}{l^2\kappa}
	\geqslant
	\frac{\kappa\lam_1}{2}\,,
\end{align}
where the constants $K, C_1, K_1, K_2$ are defined in \eqref{etabndK},  \eqref{t2mu},  \eqref{t2neq} and \eqref{ineq-alpha}; they are all independent of $\mu$ and $h$. 

For the case of stress-free boundary conditions, we assume that the following hold:
\begin{gather}
	\label{mu1-sf}
	\frac14\mup
	-\bigg(
	\frac{2g^2}{|\Omega|\kappa\epsilon_2\lam_1}
	+\frac{2g^2}{\kappa\epsilon_2}
	+\frac{\tilde{K}_1^2\epsilon_2}{\kappa l^2}
	\bigg)
	\geqslant 
	\frac{\kappa\lambda_1}{2},\\
	\label{mu2-sf}
	\frac18\mu\lam_1-\frac{|\Omega|^{-1}}{4}
	\geqslant
	0,\\
	\label{mu3-sf}
	\frac14\mup -K_{16}
	\geqslant
	\frac{\kappa\lambda_1}{4},
\end{gather}
\begin{gather}
	\label{h1-t2}
	c_1h|\Omega|^{-1/2}\leqslant\frac18,\quad 
	2c_2^2h^4\mu\lam_1|\Omega|^{-1}\leqslant \frac{1}{8},\\
	\label{h2-t2}
	\mup(c_1^2h^2+c_2h^2)\leqslant \frac{\nu}{2}\,,
\end{gather}
where the constants $\epsilon_2, \tilde{K}_1, K_{16}$, being independent of $\mu$ and $h$, are defined in \eqref{weqn2}, \eqref{K1} and \eqref{k-16}. 

The uniqueness of bounded solutions follows from Proposition \ref{cor-Lip}. In this section, we prove the existence of strong solutions.

\begin{Rmk}
	Assumptions \eqref{mu4} and \eqref{mu3-sf} are not needed for the proof of existence; they are used to prove the uniqueness of bounded solution.
\end{Rmk}


\textbf{Step 1.} Let $k$ be a fixed positive integer. For $n\geqslant r$, where $r\in\N$ is fixed in \eqref{E1}, we consider a Galerkin approximation for system \eqref{eq-aux0}:
\begin{align}\label{gal}
	\dtf{w_n}+\nu A_0w_n+P_{0,n}B_0(w_n,w_n)
	&=P_{0,n}\PL(g\eta_n \eb_2)
	-\mup P_{0,n}(\Iht w_n-v)\,,\\ \notag
	\frac{d\eta_n}{dt}
	+\kappa A_1\eta_n
	+P_{1,n}B_1(w_n,\eta_n)
	&=P_{1,n}\left(
	\frac{w_n\cdot \eb_2}{l}
	\right)\,,
\end{align}
with initial data 
\begin{align}\label{galini}
	w_n(-k(\nu\lam_1)^{-1})=0\,,\quad \eta_n(-k(\nu\lam_1)^{-1})=0,
\end{align}
where $P_{i,n}$ is the orthogonal projection onto $H_{i,n}=\mathrm{span}\{\zeta_{i,1},\cdots,\zeta_{i,n} \}$. This is a finite system of ODEs with a quadratic polynomial nonlinearity. Hence, there exists $T_n>-k(\nu\lam_1)^{-1}$, so that there exists a solution $(w_n,\eta_n)$ to the initial value problem on the interval $[-k(\nu\lam_1)^{-1}, T_n)$. 

Thanks to the initial conditions \eqref{galini}, following the 
approach used
to prove the existence and uniqueness of strong solutions for the Navier-Stokes equations and the RB system (see, e.g., \cite{constantin1988navier,temam2012infinite}), one can show by energy estimates that there exists $T_*>-\kt$, independent of $n$, such that solutions of \eqref{gal} exist on $[-\kt,T_*]$ and satisfy uniform bounds, in the relevant strong norms, which are independent of $n$. Therefore, by the Aubin-Lions compactness theorem, there exists a subsequence $\{(w_{n(j),k},\eta_{n(j),k})\}_{j=1}^\infty$ which converges to a unique strong solution $(w^{(k)},\eta^{(k)})$ to system \eqref{eq-aux0} on a common interval $[-k(\nu\lam_1)^{-1}, T_*]$ with initial data $w^{(k)}(-\kt)=0$ and $\eta^{(k)}(-\kt)=0$.
Let $[-k(\nu\lam_1)^{-1},T_{**})$ be the maximum forward interval of existence for $(w^{(k)},\eta^{(k)})$.  Note that $T_{**}\geqslant T_{*}$ and that from the above mentioned energy type estimates we have
\begin{align*}
	(w^{(k)},\eta^{(k)})\in C\Big([-k(\nu\lam_1)^{-1},T_{**}),V_0\times V_1\Big)\cap L_{\loc}^2\Big([-k(\nu\lam_1)^{-1},T_{**}),D(A_0)\times D(A_1)\Big).
\end{align*}

\textbf{Step 2.} Assume that $T_{**}<\infty$. In Section \ref{ns_bnd} and Section \ref{sf_bnd}, for the no-slip and stress-free cases respectively,
we show on the maximum interval of existence $\kintv$ for $(w^{(k)},\eta^{(k)})$  uniform (in time $t$)
bounds on the following quantities (omitting the superscript $k$ for simplicity)
\begin{gather}\label{bnds1}
	|\eta|^2,\  
	|w|^2, \ 
	\Norm{w}^2,\
	\intavee|A_0w(\tau)|^2\, d\tau\,,
	\\ \label{bndsb1}
	\Norm{\eta}^2,\ 
	\intavee|A_1\eta(\tau)|^2\,d\tau,\ 
\end{gather}
where $T:=\tunit$.


\begin{Rmk}\label{bndrmk}
	All the bounds for \eqref{bnds1} will be \emph{independent} of $k$ and $T_{**}$. On the other hand, bounds for \eqref{bndsb1} in this step may depend on $k$; we will however, improve in the next step the bounds so that they will be independent of $k$ and $T_{**}$.
\end{Rmk}

For the no-slip case, the bounds \eqref{etabndK}, \eqref{wH1nst2},  \eqref{ave_A0w0}, \eqref{etaH1bnd10b} and \eqref{ave_A1eta0} in Section \ref{ns_bnd} imply that the solution $(w^{(k)},\eta^{(k)})$ cannot blow up in the space 
\[C\big([-k(\nu\lam_1)^{-1},T_{**}),V_0\times V_1\big)\cap L_{\loc}^2\big([-k(\nu\lam_1)^{-1},T_{**}),D(A_0)\times D(A_1)\big),
\]
and thus we may extend it beyond $T_{**}$, which contradicts the maximality of $T_{**}$. Therefore, we must have $T_{**}=\infty$.

The same argument works for the stress-free case by considering the bounds \eqref{wVbnd}, \eqref{eta1bnd}, \eqref{ave_bnd_sf0},  \eqref{c50} and \eqref{Aeta_sf0} in Section \ref{sf_bnd}.

\textbf{Step 3.} 
For $\vk$, we show uniform bounds on the interval $\Ik:=[-\kt+\tunit,\infty)$, for all the quantities in \eqref{bnds1} and \eqref{bndsb1}. These bounds will all be independent of $k$. {Note that we need the extra time unit $\tunit$ in $\Ik$ due to the use of Lemma \ref{Gronwall1}.}

By Remark \ref{bndrmk}, the uniform bounds for \eqref{bnds1} in Step 2, i.e., 
\begin{enumerate}[(i)]
	\item  no-slip: \eqref{etabndK}, \eqref{wH1nst2},  \eqref{ave_A0w0};
	\item  stress-free: \eqref{wVbnd}, \eqref{eta1bnd}, \eqref{ave_bnd_sf0}, 
\end{enumerate}
are all valid on the interval $[-\kt,\infty)$ and particularly on $\Ik$; they are independent of $k$.

For the no-slip case, in subsection \ref{eta-H1-ns}, letting $\alpha_k=T=\tunit$ and $t_1=T_{**}=\infty$, by \eqref{etaH1bnd10},
we have a uniform bound on the interval $\Ik$ for $\Norm{\eta}^2$, where $C_3$ in \eqref{etaH1bnd10} is now independent of $k$. It follows that the uniform bound \eqref{ave_A1eta} is also valid for $t\in\Ik$.

The similar argument works for the stress-free case by considering   \eqref{c5} and \eqref{Aeta_sf} in 
subsection \ref{eta-H1-sf}.

\textbf{Step 4.}
For each positive integer $m$, consider a (sub)sequence of solutions
$\{(w^{(k)},\eta^{(k)})\}_{k=m+1}^\infty$. By Step 3, this sequence satisfies all the uniform bounds on \eqref{bnds1} and \eqref{bndsb1} (with $T_{**}=\infty$) on the interval $\mathcal{I}_{m+1}=[-m(\nu\lam_1)^{-1}, \infty)$, and in particular  on $[-m(\nu\lam_1)^{-1}, m(\nu\lam_1)^{-1}]$. Thus,
\begin{align}\label{ave_bndc}
	\int_{-m\tunit}^{m\tunit}|A_0w^{(k)}(\tau)|^2\,d\tau<\infty,\quad 
	\int_{-m\tunit}^{m\tunit}|{A_1\eta^{(k)}(\tau)}|^2\, d\tau<\infty\,,
\end{align}
where the bounds in \eqref{ave_bndc} may depend on $m$, but are independent of $k$. In particular, \eqref{ave_bndc} implies that
\begin{align}\label{ave_bndc2}
	\int_{-m\tunit}^{m\tunit}\left|\frac{dw^{(k)}(\tau)}{d\tau}\right|^2\,d\tau<\infty,\quad 
	\int_{-m\tunit}^{m\tunit}\left|\frac{d\eta^{(k)}(\tau)}{d\tau}\right|^2\, d\tau<\infty\,,
\end{align}
are bounded uniformly in $k$, with bounds that may depend on $m$. 

Applying the Aubin-Lions compactness theorem using  \eqref{ave_bndc}, \eqref{ave_bndc2}, and the uniform, with respect to $t$ and $k$, bounds on the quantities 
\[
|\eta^{(k)}|^2,\  
|w^{(k)}|^2, \ 
\Norm{w^{(k)}}^2,\
\Norm{\eta^{(k)}}^2,\quad t\in [-m\tunit,m\tunit],
\]
we obtain a subsequence $\{(w^{(k_{l},m)},\eta^{(k_{l},m)})\}_{l=1}^\infty$ that converges to a solution of system \eqref{eq-aux0} on the closed interval $[-m\tunit,m\tunit]$. 

We then apply the Cantor diagonal process to nested subsequences, relabeling when necessary, to get a subsequence $\{(w^{(k_{m},m)},\eta^{(k_{m},m)})\}_{m=1}^\infty$ that converges to a solution $(w,\eta)$ on $[-M\tunit,M\tunit]$ for all $M\in\N$. Note that $(w,\eta)$ is defined on $(-\infty,\infty)$. Hence, $(w,\eta)$ satisfies all the uniform bounds on \eqref{bnds1} and \eqref{bndsb1} for $t\in\R$ and thus \eqref{aux-soln}. The proof of Proposition \ref{Wmap-Exst-t1} is complete.


\subsection{No-slip BCs (bounds on $[-\kt,T_{**})$ with $T_{**}<\infty$)}\label{ns_bnd}
For simplicity, we will omit the superscript $k$ in $(w^{(k)},\eta^{(k)})$ in this section and the next (stress-free BCs). All estimates are rigorous on the maximal interval $\kintv$.

\subsubsection{Bound for $|\eta|$}
By a similar argument as in \cite[Lemma 2.1]{foias1987attractors}, we can show, by employing the maximum principle for the heat equation, that (see the Appendix)
\begin{align}\label{etabndK}
	|\eta(t)|\leqslant 2|\Omega|:=K,\quad \forall\  t\in[-\kt,T_{**}).
\end{align}

\subsubsection{Bounds for $|w|$ and $\Norm{w}$}
Taking the $L^2$ inner product of the auxiliary equation \eqref{eq-aux} with $w$ and $A_0w$ respectively, we have
\begin{gather}\label{w}
	\frac{1}{2}\frac{d}{dt}|w|^2
	+\nu\Norm{w}^2
	=g\Innerr{\eta \eb_2}{w}
	-\mup\Innerr{\Iht w-v}{w}\,,\\ \label{w2}
	\frac{1}{2}\frac{d}{dt}\Norm{w}^2+\nu\Abs{A_0w}^2
	+\Innerr{B_0(w,w)}{A_0w}
	=g\Innerr{\eta e_2}{A_0w}
	-\mup\Innerr{\Iht w-v}{A_0w}\,,
\end{gather}
where we use $b_0(w,w,w)=0$.
By the Cauchy-Schwarz, Young and Poincar\'e inequalities, we have
\begin{align}\label{t2leq1}
	-\mup\Innerr{\Iht w&-v}{w}
	\leqslant
	\mup
	\Big[
	|\Innerr{\Iht w-w}{w}|
	+|\Innerr{v}{w}|-(w,w)
	\Big]
	\\
	&\leqslant 
	\mup
	\Big[
	c_1h\Norm{w}\cdot\Abs{w}
	+c_2h^2|A_0w|\cdot|w|
	+\Abs{v}\cdot\Abs{w}
	-\Abs{w}^2
	\Big]
	\quad (\textrm{by Remark \ref{Iht}})
	\notag\\
	&\leqslant
	\mup
	\bigg[
	c_1h\lambda_1^{-1/2}\Norm{w}^2
	+2c_2^2h^4|A_0w|^2
	+2|v|^2
	-\frac{3}{4}\Abs{w}^2
	\bigg]
	\notag\\
	&\leqslant
	\frac{\nu}{4}\Norm{w}^2
	+ \frac{\nu}{8}\lambda_1^{-1}|A_0w|^2
	+2\mup |v|^2
	-\frac{3}{4}\mup|w|^2\quad (\textrm{by \eqref{muh}})\,,
	\notag
\end{align}
and 
\begin{align}\label{t2leq2}
	-&\mup\Innerr{\Iht w-v}{A_0w}
	\leqslant
	\mup
	\Big[
	|\Innerr{\Iht w-w}{A_0w}|
	+|\Innerr{v}{A_0w}|-(w,A_0w)
	\Big]
	\\ \notag
	&=	\mup
	\Big[
	|\Innerr{\Iht w-w}{A_0w}|
	+|(({v},{w}))|-(w,A_0w)
	\Big] \quad \textrm{(since $v(t)\in V_0$)}
	\\
	&\leqslant 
	\mup
	\Big[
	c_1h\Norm{w}\cdot\Abs{A_0w}
	+c_2h^2|A_0w|^2
	+\Norm{v}\cdot\Norm{w}
	\Big]
	-\mup\Norm{w}^2
	\quad (\textrm{by Remark \ref{Iht}})
	\notag\\
	&\leqslant
	\mup
	\bigg[
	c_1h\lambda_1^{-1/2}\Abs{A_0w}^2
	+c_2h^2|A_0w|^2
	+\Norm{v}^2
	-\frac{3}{4}\Norm{w}^2
	\bigg]
	\notag\\
	&\leqslant
	\frac{\nu}{8}\Abs{A_0w}^2
	+\mup \|v\|^2
	-\frac{3}{4}\mup\|w\|^2\qquad (\text{by \eqref{muh}}).
	\notag
\end{align}

For the nonlinear term, we have
\begin{align}\label{t2neq}
	|\Innerr{B_0(w,w)}{A_0w}|&\leqslant 
	\Norm{w}_{L^4}^2\Norm{\nabla w}_{L^4}^2|A_0w|
	\quad (\text{H\"{o}lder})
	\\ \notag
	&\leqslant c_L^2|w|^{1/2}\Norm{w}\cdot\Abs{A_0w}^{3/2}
	\quad (\text{Ladyzhenskaya})
	\\ \notag
	&\leqslant \frac{\nu}{8}|A_0w|^2
	+K_1|w|^2\Norm{w}^4,\quad K_1:=\frac{27c_L^8}{2\nu^3}.
\end{align}
Combining \eqref{etabndK}--\eqref{t2neq}, we get
\begin{align}
	&\frac12\dt(|w|^2+\lambda_1^{-1}\Norm{w}^2)
	+\nu (\Norm{w}^2+\lambda_1^{-1}|A_0w|^2)
	\\ \notag
	&\leqslant 
	g|\eta| |w|
	+\frac{\nu}{4}\Norm{w}^2
	+\frac{\nu}{8}\lambda_1^{-1}|A_0w|^2
	+2\mup |v|^2-\frac34\mup |w|^2
	\\ \notag
	&\quad +\lambda_1^{-1}
	\Big(
	g|\eta| |A_0w|
	+\frac{\nu}{8} |A_0w|^2	
	+\mup \Norm{v}^2
	-\frac34\mup \Norm{w}^2
	\Big)
	+\lambda_1^{-1} 
	\Big(
	\frac{\nu}{8}|A_0w|^2
	+ K_1|w|^2\Norm{w}^4
	\Big)
	\\ \notag
	&\leqslant
	\frac{g^2K}{2\nu\lambda_1} 
	+\frac{\nu\lambda_1}{2}|w|^2
	+\frac{\nu}{4}\Norm{w}^2
	+\frac{3\nu}{8}\lambda_1^{-1}|A_0w|^2
	+\frac{2g^2K}{\nu\lambda_1}+\frac{\nu}{8}\lambda_1^{-1}|A_0w|^2
	\\ \notag
	&\quad +3\mup \Norm{v}_X^2\nu^2
	-\frac34\mup (|w|^2+\lambda_1^{-1}\Norm{w}^2)
	+\lambda_1^{-1} K_1|w|^2\Norm{w}^4\,.
\end{align}
Hence, 
\begin{align}\label{t2mu}
	\frac12\dt(|w|^2+\lambda_1^{-1}\Norm{w}^2)
	&+\frac12\mup(|w|^2+\lambda_1^{-1}\Norm{w}^2)
	\\ \notag
	&+\big(
	\frac14\mup
	-K_1|w|^2\Norm{w}^2
	\big)\lambda_1^{-1}
	\Norm{w}^2
	+\frac{\nu\lam_1^{-1}}{2}|A_0w|^2
	\\ \notag
	&\leqslant 
	3\mup \Norm{v}_X^2\nu^2
	+\frac{5g^2K}{2\nu\lambda_1} 
	\\ \notag
	&\leqslant 
	\mup C_1\rho^2 \quad (\textrm{by \eqref{mu5}})\,,\quad C_1:=4\nu^2\,.
\end{align}

We now show that 
\begin{align}\label{wH1nst2}
	|w|^2+\lambda_1^{-1}\Norm{w}^2
	\leqslant
	4C_1\rho^2,\quad t\in[-\kt,T_{**})\,.
\end{align}
By continuity and the initial condition $w(-\kt)=0$, there exists $t_*\in\kintv$ such that
\begin{align*}
	|w|^2+\lambda_1^{-1}\Norm{w}^2
	\leqslant 4C_1\rho^2,\quad t\in[-\kt,t_*]\,.
\end{align*}
It then follows from \eqref{wH1nst2} and \eqref{mu3} that
\begin{align*}
	\frac14\mup 
	-K_1|w|^2\Norm{w}^2\geq 0\,,\quad t\in[-\kt,t_*]\,.
\end{align*}
Let 
\begin{align*}
	\tilde{T}
	=\sup\big\{\tau\in[-\kt,T_{**}) :  |w(t)|^2+\lambda_1^{-1}\Norm{w(t)}^2
	\leqslant 4C_1\rho^2 
	\textrm{ for all $t\in[-\kt,\tau]$}
	\big\}.
\end{align*}
Notice that $\tilde{T}\geqslant t_*>-\kt$.
We claim that $\tilde{T}=T_{**}$. If not, then $\tilde{T}<T_{**}$, and
\begin{align}\label{wnT-nst2}
	|w(\tilde{T})|^2&+\lambda_1^{-1}\Norm{w(\tilde{T})}^2= 4C_1\rho^2,\\
	\label{ineq_Aw_ns}
	\frac12\dt(|w|^2+\lambda_1^{-1}\Norm{w}^2)
	&+\frac12\mup(|w|^2+\lambda_1^{-1}\Norm{w}^2)\\\notag
	&+\frac{\nu\lam_1^{-1}}{2}|A_0w|^2
	\leqslant 
	\mup C_1\rho^2,
	\quad \forall \ t\in[-\kt,\tilde{T}].
\end{align}
Dropping the term $\frac{\nu\lam_1^{-1}}{2}|A_0w|^2$, we have by the Gronwall inequality that
\begin{align*}
	|w(\tilde{T})|^2+\lambda_1^{-1}\Norm{w(\tilde{T})}^2
	\leqslant
	2C_1\rho^2(1-e^{\mup (-\kt-\tilde{T})})
	< 2C_1\rho^2,
\end{align*}
which contradicts \eqref{wnT-nst2}. 

\subsubsection{Bound for $\int_{t}^{\min(t+T,T_{**})}|A_0w(\tau)|^2\,d\tau$}
Henceforth, we let $T=\tunit$.

Inequality \eqref{ineq_Aw_ns} implies that
\begin{align*}
	\frac12\dt(|w|^2+\lambda_1^{-1}\Norm{w}^2)
	+\frac{\nu\lam_1^{-1}}{2}|A_0w|^2
	\leqslant 
	\mup C_1\rho^2.
\end{align*}
For any $t\in\kintv$, integrating on both sides from $t$ to $\min(t+T,T_{**})$, observing that 
$
\min(t+T,T_{**})-t\leqslant T,
$
and using the bound \eqref{wH1nst2},
we have
\begin{align}\label{ave_A0we}
	{\nu}\intavee|A_0w(\tau)|^2\,d\tau
	\leqslant
	{4C_1\rho^2\lam_1}
	+T\mup^2 C_1\rho^2.
\end{align}
Since $T_{**}<\infty$, it follows that
\begin{align}\label{ave_A0w0}
	{\nu}\int_{-\kt}^{T_{**}}|A_0w(\tau)|^2\,d\tau
	<\infty.
\end{align}

\subsubsection{Bound for $\Norm{\eta}$}\label{eta-H1-ns}
Taking the $L^2$ inner product of the equation \eqref{eq-aux2} with $\eta$, and applying the Cauchy-Schwarz and Young inequalities, we have
\begin{align}\label{etaine}
	\frac{1}{2}\frac{d}{dt}|\eta|^2
	+\kappa\Norm{\eta}^2
	\leqslant
	\frac{\kappa\lambda_1}{4}|\eta|^2
	+\frac{1}{\kappa l^2\lambda_1}\Abs{w}^2.
\end{align}
Let $\ktilde=\kt$ and $\alpha_k=\frac{T_{**}+\ktilde}{2}$.
For any $t\in[-\ktilde,-\ktilde+\alpha_k)$, integrating \eqref{etaine} from $t$ to $t+\alpha_k$, we have
\begin{align}\label{etaH1ave1}
	{\kappa}\int_{t}^{t+\alpha_k}\Norm{\eta(\tau)}^2\ d\tau
	&\leqslant 
	\frac{K^2}{2}+
	\alpha_k
	\left(
	\frac{\kappa \lambda_1K^2}{4\rho^2}
	+\frac{4C_1}{\kappa l^2\lambda_1}
	\right)
	\rho^2=:\beta_k.
\end{align}

By taking the $L^2$ inner product of the equation (\ref{eq-aux2}) with $A_1\eta$, we have 
\begin{align}\label{eta-ineq}
	\frac{1}{2}\frac{d}{dt}\Norm{\eta}^2
	+\kappa|A_1\eta|^2
	+(B_1(w,\eta),A_1\eta)=\frac{(w\cdot \eb_2, A_1\eta)}{l}
	\leqslant
	\frac{\kappa}{4}|A_1\eta|^2+ \frac{1}{l^2\kappa}|w|^2.
\end{align}
Integrating by parts, we have (as in \cite[(3.22)]{Farhat2015continuous})
\begin{align}
	\Abs{\Innerr{B_1(w,\eta)}{A_1\eta}}
	&\leqslant
	\Norm{w}\cdot\Norm{\nabla\eta}_{L^4}^2 
	\quad {\textrm{(H\"{o}lder)}}
	\\ \notag
	& \leqslant
	c_L\Norm{w}\cdot\Norm{\eta}\cdot|A_1\eta|\quad (\textrm{Ladyzhenskaya})
	\\ \notag
	&
	\leqslant
	\frac{c_L^2}{\kappa}\Norm{w}^2\Norm{\eta}^2
	+\frac{\kappa}{4}\Abs{A_1\eta}^2.
\end{align}
Consequently, 
\begin{gather}\label{etaH1ineq}
	\dt\Norm{\eta}^2+\kappa\Abs{A_1\eta}^2
	\leqslant
	\frac{2c_L^2}{\kappa}\Norm{w}^2\Norm{\eta}^2
	+\frac{2}{l^2\kappa}\Abs{w}^2
	\leqslant
	\frac{8c_L^2 C_1\lambda_1}{\kappa}\rho^2\Norm{\eta}^2
	+\frac{8C_1}{l^2\kappa}\rho^2\,.
\end{gather}

We now recall the following uniform Gronwall inequality from \cite{foias1987attractors}. 
\begin{Lem}[Uniform Gronwall]\label{Gronwall1}
	Let $g$, $h$, $y$ be three positive locally integrable functions on $(t_0,t_1)$ which satisfy for all $t$ with $t_0\leqslant t<t+\alpha< t_1$,
	\begin{align*}
		\frac{dy}{dt}\leqslant gy+h,\quad
		\int_{t}^{t+\alpha}g(\tau)\ d\tau\leqslant a_1,\quad
		\int_{t}^{t+\alpha}h(\tau)\ d\tau\leqslant a_2,\quad
		\int_{t}^{t+\alpha}y(\tau)\ d\tau\leqslant a_3,
	\end{align*}
	where $a_1, a_2, a_3, \alpha$ are positive constants. Then
	\begin{align*}
		y(t+\alpha)\leqslant
		\left(\frac{a_3}{\alpha}+a_2\right)e^{a_1},\quad t_0\leqslant t< t+\alpha<t_1\,.
	\end{align*}
\end{Lem}

Applying  Lemma \ref{Gronwall1} to \eqref{etaH1ineq} with 
\begin{gather*}
	t_0=-\kt,\quad
	t_1=T_{**}, \quad
	\alpha = \alpha_k,
	\\
	g(t) = \frac{8c_L^2C_1\lam_1\rho^2}{\kappa},\quad
	h(t) = \frac{8C_1\rho^2}{l^2\kappa},\quad
	y(t)=\Norm{\eta(t)}^2,
	\\
	a_1=\frac{8c_L^2 C_1\lambda_1}{\kappa}\rho^2 \alpha,\quad
	a_2=\frac{8C_1}{l^2\kappa}\rho^2 \alpha,\quad
	a_3=\frac{\beta_k}{\kappa} \alpha \,,
\end{gather*}
we get
\begin{align}\label{etaH1bnd10}
	\sup_{t\in[-\tilde{k}+\alpha,T_{**})}\Norm{\eta(t)}^2\leqslant
	\left(
	\frac{a_3}{\alpha}+a_2
	\right)e^{a_1}=:C_3\,,
\end{align}
and thus
\begin{align}\label{etaH1bnd10b}
	\sup_{t\in[-\tilde{k},T_{**})}\Norm{\eta(t)}^2\leqslant
	\left(
	\frac{a_3}{\alpha}+a_2
	\right)
	e^{a_1}
	+\sup_{t\in[-\tilde{k},-\tilde{k}+\alpha]} \Norm{\eta(t)}^2
	<\infty.
\end{align}


\subsubsection{Bound for $\intavee|A_1\eta(\tau)|^2\,d\tau$}\label{Aeta}

For any $t\in[-\ktilde+\alpha_k,T_{**})$, inserting the bound \eqref{etaH1bnd10} in \eqref{etaH1ineq} and then integrating from $t$ to $\min(t+T,T_{**})$ on both sides, we have
\begin{gather}\label{ave_A1eta}
	{\kappa}\intavee\Abs{A_1\eta(\tau)}^2\, d\tau
	\leqslant
	{ C_3}
	+
	\left(
	\frac{8c_L^2 C_1\lambda_1 C_3}{\kappa}
	+\frac{8C_1}{l^2\kappa}
	\right)
	\rho^2T\,.
\end{gather}
Since $T_{**}<\infty$, it follows that
\begin{gather}\label{ave_A1eta0}
	\int_{-\kt}^{T_{**}}\Abs{A_1\eta(\tau)}^2\, d\tau
	<\infty.
\end{gather}

\subsection{Stress-free BCs (bounds on $[\kt,T_{**})$ with $T_{**}<\infty$)}\label{sf_bnd}
The argument using the maximum principle for showing the bound for $|\eta|$ in Section \ref{ns_bnd} also works here. Taking advantage of the orthogonality property that $b_0(w,w,A_0w)=0$ in the case of stress-free BCs, we combine the estimates of $\VNorm{w}$ and $|\eta|$ together.

\subsubsection{Bounds for $\|w\|_{V_0}$ and $|\eta|$ }
Taking the $L^2$ inner products of the auxiliary system (\ref{eq-aux0}) with $w$, $A_0w$ and $\eta$ repectively, we have
\begin{align}\label{weqn1}
	\epsilon_1\left(
	\frac12\dt |w|^2+\nu \Norm{w}^2
	\right)
	&=
	\epsilon_1
	\left(
	g(\eta \eb_2, w)
	-\mup(\Iht w-v, w)
	\right),\quad \epsilon_1:=\frac{1}{|\Omega|},
	\\
	\label{weqn}
	\frac12\dt\Norm{w}^2+\nu |A_0w|^2
	&=
	g(\eta \eb_2, A_0w)
	-\mup(\Iht w-v,A_0w),
	\\
	\label{weqn2}
	\epsilon_2\left(\frac12\dt|\eta|^2+\kappa\Norm{\eta}^2\right)
	&=
	\epsilon_2\left(\frac{(w\cdot \eb_2,\eta)}{l}\right),\quad 
	\epsilon_2 := (\nu\lambda_1)^2,
\end{align}
where we used $b_0(w,w,w)=0$, $b_0(w,w,A_0w)=0$ and $b_1(w,\eta,\eta)=0$. Note that equations \eqref{weqn1}--\eqref{weqn2} have the same dimension and no nonlinear term appears in the equations above. 

Now we estimate the right-hand sides of the three equations above as follows:
\begin{align}\label{wineq1}
	-&{\mup}{\epsilon_1}(\Iht w-v, w) 
	\leqslant 
	{\mup}{\epsilon_1}
	\left(|(\Iht w-w,w)|
	+|(v,w)|
	-(w,w)\right)
	\\ \notag
	&\leqslant
	\frac{\mup}{|\Omega|}
	\left(c_1h\VNorm{w}\cdot|w|
	+c_2h^2|A_0w|^2\cdot|w|
	+\Abs{v}^2
	+\frac{1}{4}\Abs{w}^2
	-\Abs{w}^2\right)
	\quad(\textrm{by Remark \ref{Iht}})
	\\ \notag
	&\leqslant 
	\frac{\mup}{|\Omega|}
	\left(c_1h|\Omega|^{1/2}\VNorm{w}^2
	+2c_2^2h^4|A_0w|
	+\frac18|w|^2
	+\Abs{v}^2
	+\frac{1}{4}\Abs{w}^2
	-\Abs{w}^2\right)
	\quad (\textrm{by \eqref{u-Poin}})
	\\ \notag
	&\leqslant
	\frac18\mup\VNorm{w}^2
	+\frac\nu8|A_0w|^2
	+\mup{\epsilon_1}\Abs{v}^2
	-\frac34{\mup}{\epsilon_1}|w|^2
	+\frac18\mup\epsilon_1|w|^2
	\quad 
	(\text{by \eqref{h1-t2}})
	\\ \notag
	&\leqslant
	\frac14\mup\VNorm{w}^2
	+\mup{\epsilon_1}\Abs{v}^2
	-\frac34{\mup}{\epsilon_1}|w|^2
	+\frac\nu8|A_0w|^2
\end{align}
\begin{align}\label{wineq2}
	-\mup&(\Iht w-v, A_0w) 
	\leqslant 
	\mup
	\left(
	|(\Iht w-w,A_0w)|
	+|(v,A_0w)|
	-(w,A_0w)
	\right)
	\\ \notag
	&\leqslant
	\mup
	\left(c_1h\VNorm{w}\cdot|A_0w|
	+c_2h^2|A_0w|^2
	+\Norm{v}^2
	+\frac{1}{4}\Norm{w}^2
	-\Norm{w}^2\right)
	\quad(\textrm{by Remark \ref{Iht}})
	\\ \notag
	&\leqslant 
	\mup
	\left(\frac14\VNorm{w}^2
	+c_1^2h^2|A_0w|^2
	+c_2h^2|A_0w|^2
	+\Norm{v}^2
	+\frac{1}{4}\Norm{w}^2
	-\Norm{w}^2\right)\\ \notag
	&\leqslant
	\frac14\mup\VNorm{w}^2
	+\frac{\nu}{8}|A_0w|^2
	+\mup\Norm{v}^2
	-\frac34\mup\Norm{w}^2 
	\quad 
	(\text{by \eqref{h2-t2}})\,.
\end{align}
%
\begin{align}\label{wineq3}
	{\epsilon_1}|(g\eta \eb_2,w)|
	&\leqslant
	\frac{g}{|\Omega|}|\eta|\cdot|w|
	\leqslant
	\frac{g\lambda_1^{-1/2}}{|\Omega|}\Norm{\eta}\cdot|w|\\\notag
	&\leqslant 
	\frac{\kappa\epsilon_2}{8}\Norm{\eta}^2
	+\frac{2}{\kappa\epsilon_2}\frac{g^2\lambda_1^{-1}}{|\Omega|^2}|w|^2
	\leqslant
	\frac{\kappa\epsilon_2}{8}\Norm{\eta}^2
	+\frac{2}{\kappa\epsilon_2}
	\frac{g^2\lambda_1^{-1}}{|\Omega|}\VNorm{w}^2,
\end{align}
\begin{align}\label{wineq4}
	|(g\eta \eb_2, A_0w)|
	\leqslant
	g\Norm{\eta \eb_2}\cdot\Norm{w}
	\leqslant
	\frac{\kappa \epsilon_2}{8}\Norm{\eta}^2
	+\frac{2g^2}{\kappa \epsilon_2}\Norm{w}^2
	\leqslant
	\frac{\kappa \epsilon_2}{8}\Norm{\eta}^2
	+\frac{2g^2}{\kappa \epsilon_2}\VNorm{w}^2,
\end{align}
\begin{align}\label{wineq5}
	\frac{\epsilon_2}{l}|(w\cdot \eb_2,\eta)|
	\leqslant
	\frac{\epsilon_2}{l} |w\cdot \eb_2|\cdot |\eta|
	\leqslant
	\frac{\tilde{K}_1\epsilon_2}{l}\VNorm{w}\Norm{\eta} 
	\leqslant
	\frac{\kappa\epsilon_2}{4}\Norm{\eta}^2
	+\frac{\tilde{K}_1^2\epsilon_2}{\kappa l^2}\VNorm{w}^2,
\end{align}
where
\begin{align}\label{K1}
	\tilde{K}_1:=|\Omega|^{1/2}\lambda_1^{-1/2}.
\end{align}

Combining (\ref{weqn1})--\eqref{wineq5}, we have
\begin{align}\label{w_energy_ineq}
	&\ \ \ \ \frac12\dt\big(
	\epsilon_1|w|^2
	+\Norm{w}^2+\epsilon_2|\eta|^2
	\big)
	+\epsilon_1\nu\Norm{w}^2
	+\nu|A_0w|^2
	+{\kappa\epsilon_2}\Norm{\eta}^2
	\\ \notag
	&\leqslant
	\frac12\mup\VNorm{w}^2
	-\frac34\mup(\epsilon_1|w|^2+\Norm{w}^2)\\ \notag
	&\ \ \ \ +
	\bigg(
	\frac{2g^2\lambda_1^{-1}}{|\Omega|\kappa\epsilon_2}
	+\frac{2g^2}{\kappa\epsilon_2}
	+\frac{\tilde{K}_1^2\epsilon_2}{\kappa l^2}
	\bigg) 
	\VNorm{w}^2
	+\mup(\epsilon_1|v|^2+\Norm{v}^2)\\ \notag
	&\ \ \ \ +
	\frac12\kappa\epsilon_2\Norm{\eta}^2
	+\frac{\nu}{2}|A_0w|^2\,,
\end{align}
and thus, after dropping nonnegative terms on the left,
\begin{align*}
	&\ \ \ \ \frac12\dt\big(
	\VNorm{w}^2+\epsilon_2|\eta|^2
	\big)
	+
	\VNorm{w}^2
	\left(
	\frac14\mup
	-\bigg(
	\frac{2g^2\lambda_1^{-1}}{|\Omega|\kappa\epsilon_2}
	+\frac{2g^2}{\kappa\epsilon_2}
	+\frac{\tilde{K}_1^2\epsilon_2}{\kappa l^2}
	\bigg)
	\right)+
	\frac{\kappa\lambda_1}{2}\cdot \epsilon_2\Abs{\eta}^2 \\ \notag
	&\leqslant
	\mup\Norm{v}_X^2\nu^2\lambda_1\,,
\end{align*}
By (\ref{mu1-sf}), 
we have
\begin{align*}
	\dt\big(
	\VNorm{w}^2+\epsilon_2 \Abs{\eta}^2
	\big)
	+
	\big(\VNorm{w}^2+\epsilon_2\Abs{\eta}^2\big)
	\cdot 
	(\lambda_1\kappa)
	\leqslant
	2\mup\Norm{v}_X^2\nu^2\lambda_1\,,
\end{align*}
which implies by the Gronwall inequality that
\begin{align}\label{wbnd0}
	\VNorm{w}^2+(\nu\lambda_1)^2|\eta|^2
	\leqslant
	\frac{2\mup}{\lambda_1\kappa}
	\Norm{v}_X^2\nu^2\lambda_1\,,
\end{align}
and in particular
\begin{align}\label{etabnd0}
	|\eta|^2\leqslant \tilde C_0 \mu \Norm{v}_X^2\,,\quad 
	\tilde C_0 := 
	\frac{2\nu\lambda_1\nu^2\lambda_1}{\lambda_1\kappa(\nu\lambda_1)^2}
	=\frac{2\nu}{\lambda_1\kappa}\,.
\end{align}

We use \eqref{etabnd0} to improve the bound on $\VNorm{w}^2$. Instead of \eqref{wineq3} and \eqref{wineq4}, we now estimate as follows
\begin{align}\label{wineq3b}
	{\epsilon_1}|(g\eta \eb_2,w)|
	&\leqslant
	g\epsilon_1|\eta||w|
	\leqslant
	\frac{g^2}{\nu}|\eta|^2+\frac{\nu}{4}\epsilon_1^2|w|^2,
\end{align}
\begin{align}\label{wineq4b}
	|(g\eta \eb_2, A_0w)|
	\leqslant
	g|\eta||A_0w|
	\leqslant
	\frac{g^2}{\nu}|\eta|^2+\frac{\nu}{4}|A_0w|^2.
\end{align}
Combining \eqref{weqn1}, \eqref{weqn}, \eqref{wineq1}, \eqref{wineq2}, \eqref{wineq3b} and \eqref{wineq4b}, we have
\begin{align}\label{wineq6}
	&\ \ \ \ \frac12 \dt \big(
	\epsilon_1|w|^2
	+\Norm{w}^2\big)
	+{\epsilon_1}\nu\Norm{w}^2
	+\nu|A_0w|^2\\ \notag
	&\leqslant 
	\big(
	\frac12\mup-\frac34\mup
	\big)\VNorm{w}^2
	+\mup\VNorm{v}^2
	+\frac{\nu}{4}|A_0w|^2 
	\\ \notag
	& \ \ \ \ 
	+\frac{g^2}{\nu}|\eta|^2
	+\frac{\nu}{4}\epsilon_1^2|w|^2
	+\frac{g^2}{\nu}|\eta|^2
	+\frac{\nu}{4}|A_0w|^2,
\end{align}
which implies that
\begin{align*}
	\frac12\dt \VNorm{w}^2+
	\VNorm{w}^2\big(
	\frac14\mup
	-\frac{\epsilon_1\nu}{4}
	\big)
	+\frac{\nu}{2}|A_0w|^2
	\leqslant
	\frac{2g^2}{\nu}|\eta|^2
	+\mup \VNorm{v}^2.
\end{align*}
Therefore, 
by \eqref{mu2-sf},
\begin{align}\label{wineq7}
	\dt \VNorm{w}^2
	+\frac14\mup\VNorm{w}^2
	+{\nu}|A_0w|^2
	\leqslant 
	2\mu\bigg(
	\frac{2g^2\tilde C_0}{\nu}
	+\nu\lambda_1\nu^2\lambda_1
	\bigg)\Norm{v}_X^2.
\end{align}
Dropping the term ${\nu}|A_0w|^2$ in \eqref{wineq7} and using the Gronwall inequality, we conclude that
\begin{align}\label{wVbnd}
	\VNorm{w}^2
	\leqslant
	\tilde C_1\Norm{v}_{X}^2
\end{align}
where 
\begin{align}\label{c1}
	\tilde C_1 := \frac{2\mu\bigg(
		\frac{2g^2\tilde C_0}{\nu}
		+\nu\lambda_1\nu^2\lambda_1
		\bigg)
	}{\frac14\mup}
	= \frac{32 g^2}{\lambda_1\kappa\nu\lambda_1}
	+8\nu^2\lambda_1.
\end{align}
Note that the constant $\tilde C_1$ is independent of $\mu$.

By \eqref{weqn2} and \eqref{wineq5}, we have
\begin{align*}
	\frac12\dt |\eta|^2
	+\kappa\Norm{\eta}^2
	\leqslant
	\frac{\kappa}{4}\Norm{\eta}^2
	+\frac{\tilde K_1^2}{\kappa l^2}\VNorm{w}^2
\end{align*}
and thus by (\ref{wVbnd}) and the Poincar\'{e} inequality,
\begin{align*}
	\frac12\dt |\eta|^2+\frac{\kappa\lambda_1}{2}\Abs{\eta}^2\leqslant
	\tilde K_2\Norm{v}_X^2\,,\quad \tilde K_2 := 
	\frac{\tilde K_1^2\tilde C_1}{\kappa l^2}.
\end{align*}
Consequently, by the Gronwall inequality again, we have
\begin{align}\label{eta1bnd}
	|\eta|^2\leqslant \tilde C_2\Norm{v}_X^2\,,\quad \tilde C_2:=\frac{2\tilde K_2}{\lambda_1\kappa}\,,
\end{align}
where $\tilde C_2$ is also independent of $\mu$.

\subsubsection{Bound for $\intavee|A_0w(\tau)|^2\,d\tau$}
For any $t\in\kintv$, dropping the term $\frac14\mup\VNorm{w}^2$ in \eqref{wineq7} and integrating, then using the bound \eqref{wVbnd}, we have

\begin{align}\label{ave_bnd_sf}
	{\nu}
	\intavee|A_0w(\tau)|^2\,d\tau
	\leqslant 
	\tilde C_1\Norm{v}_{X}^2
	+
	2\mu T\bigg(
	\frac{2g^2\tilde C_0}{\nu}
	+\nu\lambda_1\nu^2\lambda_1
	\bigg)\Norm{v}_X^2.
\end{align}
Since $T_{**}<\infty$, it follows that
\begin{align}\label{ave_bnd_sf0}
	{\nu}
	\int_{-\kt}^{T_{**}}|A_0w(\tau)|^2\,d\tau
	<\infty.
\end{align}

\subsubsection{Bound for $\Norm{\eta}$}\label{eta-H1-sf}

Proceeding as in the no-slip case (Section \ref{eta-H1-ns}) but using the bounds \eqref{wVbnd} and \eqref{eta1bnd} for $|w|$ and $|\eta|$ in \eqref{etaine} instead, we get for any $t\in[-\ktilde,-\ktilde+\alpha_k)$, $\ktilde=\kt$ and $\alpha_k=\frac{T_{**}+\ktilde}{2}$,
\begin{align}\label{etaH1ave-sf}
	\kappa\int_t^{t+\alpha_k}\Norm{\eta(\tau)}^2\ d\tau
	&\leqslant 
	\frac{\tilde C_2}{2}
	+\alpha_k
	\left(
	\frac{\kappa \lambda_1\tilde C_2}{4}
	+\frac{\tilde C_1}{\kappa l^2\lambda_1}
	\right)\NormX{v}^2=:\tilde\beta_k
\end{align}


Similarly as in \eqref{eta-ineq}, we have
\begin{align}\label{eta1}
	\frac{1}{2}\frac{d}{dt}\Norm{\eta}^2
	&+\kappa|A_1\eta|^2
	+(B_1(w,\eta),A_1\eta)
	\\\notag
	&\leqslant
	\frac{\kappa}{4}|A_1\eta|^2+ \frac{1}{l^2\kappa}|w|^2
	\leqslant
	\frac{\kappa}{4}|A_1\eta|^2
	+\frac{|\Omega|}{l^2\kappa}\VNorm{w}^2\\ \notag
	&\leqslant
	\frac{\kappa}{4}|A_1\eta|^2
	+\frac{|\Omega|\tilde C_1}{l^2\kappa}\Norm{v}_X^2.
\end{align}
For the nonlinear term, we have
\begin{align}\label{eta2}
	|(B_1(w,\eta),A_1\eta)|
	&\leqslant 
	|A_1\eta|\cdot\Norm{w}_{L^4}\Norm{\nabla\eta}_{L^4}\quad
	\text{(H\"{o}lder)}\\ \notag
	&\leqslant
	c_L|A_1\eta|\cdot|w|^{1/2}\VNorm{w}^{1/2}\Norm{\eta}^{1/2}|A_1\eta|^{1/2}
	\quad\text{(Ladyzhenskaya)}\\ \notag
	&\leqslant
	c_L|\Omega|^{1/2}\VNorm{w}|A_1\eta|^{3/2}\Norm{\eta}^{1/2}
	\quad\text{(by (\ref{u-Poin}))}
	\\\notag
	&\leqslant
	c_L|\Omega|^{1/2}\tilde C_1^{1/2}\NormX{v}|A_1\eta|^{3/2}\Norm{\eta}^{1/2}
	\\\notag
	&\leqslant
	\frac{\kappa}{4}|A_1\eta|^2
	+{\tilde K_3}\NormX{v}^4\Norm{\eta}^2
	\quad \textrm{(Young)}\,\quad 	\tilde K_3:=\frac{27}{4\kappa^3}c_L^4|\Omega|^2\tilde C_1^2\,.
\end{align}

By (\ref{eta1}) and (\ref{eta2}), we have 
\begin{gather}\label{A1etaineq1}
	\frac{d}{dt}\Norm{\eta}^2+\kappa\Abs{A_1\eta}^2
	\leqslant
	{2\tilde K_3}\NormX{v}^4
	\Norm{\eta}^2
	+\tilde K_4\Norm{v}_X^2,
	\quad \tilde K_4
	:=\frac{2|\Omega|\tilde C_1}{l^2\kappa}.
\end{gather}
Proceeding as in Section \ref{eta-H1-ns}, using Lemma \ref{Gronwall1} with 
\begin{gather*}
	t_0=-\kt,\quad
	t_1=T_{**}, \quad
	\alpha = \alpha_k,
	\\
	g(t)={2K_3\rho^4},\quad
	h(t)=K_4\rho^2, \quad
	y(t)=\Norm{\eta(t)}^2,
	\\
	a_1:={2\tilde K_3\alpha}\rho^4,\quad 
	a_2:=\tilde K_4\rho^2\alpha,\quad
	a_3:= \frac{\tilde\beta_k}{\kappa}\alpha\,,
\end{gather*}
we get
\begin{align}\label{c5}
	\sup_{t\in[-\ktilde+\alpha,T_{**})}\Norm{\eta(t)}^2
	\leqslant \tilde C_4\Norm{v}_X^2\,,
	\quad
	\tilde C_4:=(\tilde C_3+\tilde K_4T)e^{2\tilde K_3T\rho^4}\,,
\end{align}
and as in Section \eqref{ns_bnd}, 
\begin{align}\label{c50}
	\sup_{t\in[-\ktilde,T_{**})}\Norm{\eta(t)}^2
	<\infty
	\,.
\end{align}

\subsubsection{Bound for  $\intavee|A_1\eta(\tau)|^2\,d\tau$}
Similarly as in Section \ref{Aeta}, combining \eqref{A1etaineq1} and \eqref{c5},  we get for any $t\in[-\ktilde+\alpha_k,T_{**})$,
\begin{align}\label{Aeta_sf}
	\kappa
	\intavee\Abs{A_1\eta(\tau)}^2\,d\tau
	\leqslant
	\tilde C_4\Norm{v}_X^2
	+T\bigg({2\tilde K_3}\NormX{v}^4
	\tilde C_4\Norm{v}_X^2
	+\tilde K_4\Norm{v}_X^2
	\bigg).
\end{align}
Also, 
\begin{align}\label{Aeta_sf0}
	\kappa
	\int_{-\kt}^{T_{**}}\Abs{A_1\eta(\tau)}^2\,d\tau
	<\infty.
\end{align}

\section{Lipchitz Property of the map $\tdW$}\label{varphibnd}
We assume in this section that $\NormX{v_i}\leqslant\rho$, $i=1,2$. Let $\varphi=w_1-w_2$, $\psi=\eta_1-\eta_2$ and $\gamma=v_1-v_2$ where $(w_i,\eta_i)=\widetilde W(v_i)$. We establish in this section the Lipchitz property of the map $\tdW$ for each set of boundary conditions.

By the auxiliary system \eqref{eq-aux0}, we have
\begin{align}\label{eq3-ch3-typeI}
	&\frac{d\varphi}{dt}+\nu A_0\varphi+B_0(w_2,\varphi)+B_0(\varphi,w_1)
	=\PL(g\psi \eb_2)-\mup (\Iht \varphi-\gamma)\,,\\
	\label{eq4-ch3-typeI}
	&\frac{d\psi}{dt}+\kappa A_1\psi +B_1(w_2,\psi)+B_1(\varphi,\eta_1)=\frac{\varphi \cdot \eb_2}{l}\,.
\end{align}

\subsection{No-slip BCs}
\subsubsection{Bound for $\Norm{\varphi}^2$ and $|\psi|^2$ by $\NormX{\gamma}^2$ }
Taking the $L^2$ inner product  of (\ref{eq3-ch3-typeI})--(\ref{eq4-ch3-typeI}) with $A_0\varphi$ and $\psi$ respectively, we have 
\begin{align}\label{eq1-ch6-typeI}
	&\ \ \ \frac12\frac{d}{dt}\|\varphi\|^2
	+\nu|A_0\varphi|^2
	+(B_0(w_2,\varphi),A_0\varphi)
	+(B_0(\varphi,w_1),A_0\varphi)\\
	&=
	(g\psi \eb_2,A_0\varphi)
	-\mup(\Iht \varphi-\gamma,A_0\varphi),\notag\\
	\label{eq2-ch6-typeI}
	&\ \ \  \frac{1}{2}\frac{d}{dt}\Abs{\psi}^2
	+\kappa\Norm{\psi}^2
	+(B_1(\varphi,\eta_1),\psi)
	=\frac{1}{l}(\varphi \eb_2,\psi).
\end{align}
Proceeding as for \eqref{t2leq2}, we find
\begin{align}\label{linear-eq1-ch3}
	-\mup(\Iht \varphi-\gamma,A_0\varphi)
	\leqslant
	\frac{\nu}{8}|A_0\varphi|^2
	+\frac12\mup \Norm{\gamma}^2
	-\frac12 \mup \Norm{\varphi}^2 \,.
\end{align}
By the Cauchy-Schwarz, Young and Poincar\'e inequalities, we have
\begin{align} \label{linear-eq2-ch3}
	&\big(g\psi \eb_2,A_0\varphi\big)
	\leqslant 
	g\Norm{\psi}\cdot\Norm{\varphi}
	\leqslant 
	\frac{\kappa(\nu\lambda_1)^2}{4}\Norm{\psi}^2
	+\frac{g^2}{\kappa(\nu\lambda_1)^2}\Norm{\varphi}^2\,,
\end{align}
\begin{align} \label{linear-eq3-ch3-typeI}
	\frac{1}{l}(\varphi \cdot \eb_2,\psi)
	\leqslant
	\frac{1}{l} |\varphi|\cdot|\psi|
	\leqslant 	
	\frac{1}{l\lambda_1}\Norm{\varphi}\cdot\Norm{\psi}
	\leqslant
	\frac{\kappa}{8}\Norm{\psi}^2
	+\frac{2}{l^2\lambda_1^2\kappa}\Norm{\varphi}^2\,.
\end{align}

For the two nonlinear terms involving $B_0$, we have (see \cite{Titi1987Criterion})
\begin{align}
	|(B_0(w_2,\varphi),A_0\varphi)|
	&\leqslant
	c_T\Norm{w_2}
	\cdot\Norm{\varphi}
	\left(
	\log\frac{e|A_0\varphi|}{\lambda_1^{1/2}\Norm{\varphi}}
	\right)^{1/2}
	|A_0\varphi|\\
	&\leqslant
	\frac{c_T^2}{\nu}\Norm{w_2}^2\Norm{\varphi}^2
	\log\frac{e|A_0\varphi|}{\lambda_1^{1/2}\Norm{\varphi}}
	+\frac{\nu}{4}|A_0\varphi|^2\notag
\end{align}
and by the Br\'ezis-Gallouet inequality (see \cite{Brezis1980Nonlinear,Titi1987Criterion})
\begin{align}
	|(B_0(\varphi,w_1),A_0\varphi)|
	&\leqslant
	c_B\Norm{w_1}
	\cdot\Norm{\varphi}
	\left(
	\log\frac{e|A_0\varphi|}{\lambda_1^{1/2}\Norm{\varphi}}
	\right)^{1/2}
	|A_0\varphi|\\
	&\leqslant 
	\frac{c_B^2}{\nu}\Norm{w_1}^2\Norm{\varphi}^2
	\log\frac{e|A_0\varphi|}{\lambda_1^{1/2}\Norm{\varphi}}
	+\frac{\nu}{4}|A_0\varphi|^2\notag
\end{align}

For the nonlinear term involving $B_1$, we have
\begin{align}\label{eta-ineq2}
	|(B_1(\varphi,\eta_1),\psi)|
	&\leqslant
	\Norm{\varphi}_{L^4}\Norm{\psi}_{L^4}\Norm{\eta_1}
	\leqslant
	c_L|\varphi|^{1/2}\Norm{\varphi}^{1/2}|\psi|^{1/2}\Norm{\psi}^{1/2}\Norm{\eta_1}\\
	&\leqslant
	\frac{c_L}{\sqrt{\lambda_1}}\Norm{\varphi}
	\Norm{\psi}
	\Norm{\eta_1}
	\leqslant
	\frac{\kappa}{8}\Norm{\psi}^2
	+\frac{2c_L^2}{\kappa\lambda_1}\Norm{\varphi}^2\Norm{\eta_1}^2\,.
	\notag
\end{align}

Combining the estimates above, we have for $\Norm{\varphi}$,
\begin{align}\label{dif-ineq1}
	\frac{1}{2}\frac{d}{dt}\Norm{\varphi}^2
	&+\Norm{\varphi}^2\bigg[
	\frac12\mup
	-\frac{g^2}{\kappa(\nu\lambda_1)^2}\\
	&+\frac{\nu|A_0\varphi|^2}{4\Norm{\varphi}^2}
	-(c_T^2\Norm{w_2}^2+c_B^2\Norm{w_1}^2)
	\left(
	\log\frac{e|A_0\varphi|}{\lambda_1^{1/2}\Norm{\varphi}}
	\right)
	\nu^{-1}
	\bigg]\notag\\
	&+\nu|A_0\varphi|^2
	\left[1-\frac18
	-\frac{1}{2}-\frac14
	\right]
	-\frac{\kappa(\nu\lambda_1)^2}{4}\Norm{\psi}^2\notag\\
	&\leqslant 
	\frac{\mup}{2}\Norm{\gamma}^2.\notag
\end{align}
But the second line of \eqref{dif-ineq1} can be estimated by
\begin{align}
	\frac{\nu|A_0\varphi|^2}{4\Norm{\varphi}^2}
	&-(c_T^2\Norm{w_2}^2+c_B^2\Norm{w_1}^2)
	\left(
	\log\frac{e|A_0\varphi|}{\lambda_1^{1/2}\Norm{\varphi}}
	\right)
	\nu^{-1}\\
	&\geqslant
	\frac{\nu|A_0\varphi|^2}{4\Norm{\varphi}^2}
	-\frac{c_T^2\Norm{w_2}^2+c_B^2\Norm{w_1}^2}{\nu}
	\left(
	1+2\log\frac{|A_0\varphi|}{\lambda_1^{1/2}\|\varphi\|}
	\right)
	\notag\\
	&=
	\frac{\lambda_1\nu}{4}
	\bigg[
	\frac{|A_0\varphi|^2}{\lambda_1\Norm{\varphi}^2}
	-\frac{c_T^2\Norm{w_2}^2+c_B^2\Norm{w_1}^2}{\nu^2\lambda_1/4}
	\left(
	1+\log\frac{|A_0\varphi|^2}{\lambda_1\|\varphi\|^2}
	\right)
	\bigg]
	\notag\\
	&\geqslant
	\frac{\lambda_1\nu}{4}(-\epsilon\log\epsilon)\notag
\end{align}
where we used the elementary relation (see \cite[p.371]{foias2014unified})
\begin{align}
	\chi-\epsilon(1+\log \chi)\geqslant -\epsilon\log\epsilon,\quad \forall\, \chi \geqslant 1,
\end{align}
with
\begin{align}\label{ineq-alpha}
	\epsilon:=\frac{c_T^2\Norm{w_2}^2+c_B^2\Norm{w_1}^2}{\nu^2\lambda_1/4}
	\leqslant 
	\frac{4(c_T^2+c_B^2)\rho^2}{\nu^2\lambda_1}=:K_2\, .
\end{align}
Hence, 
\begin{align}\label{dif-ineq2}
	\frac{1}{2}\frac{d}{dt}\Norm{\varphi}^2
	&+\Norm{\varphi}^2
	\bigg[
	\frac12\mup
	-\frac{g^2}{\kappa(\nu\lambda_1)^2}
	-\frac{\lambda_1\nu}{4}(K_2\log K_2)
	\bigg]\\
	&+
	\frac{\nu}{8}|A_0\varphi|^2
	-\frac{\kappa(\nu\lambda_1)^2}{4}\Norm{\psi}^2\notag\\
	&\leqslant 
	\frac{\mup}{2}\Norm{\gamma}^2.\notag
\end{align}

Combining \eqref{eq2-ch6-typeI}, \eqref{linear-eq3-ch3-typeI} and \eqref{eta-ineq2}, we have
\begin{align}\label{dif-ineq3}
	\frac12\frac{d}{dt}|\psi|^2
	-\Norm{\varphi}^2
	\left[
	\frac{2c_L^2}{\kappa\lambda_1^2}\Norm{\eta_1}^2
	+\frac{2}{l^2\lambda_1^2\kappa}
	\right]
	+ \left[
	\kappa\Norm{\psi}^2
	-\frac{\kappa}{8}\Norm{\psi}^2
	-\frac{\kappa}{8}\Norm{\psi}^2
	\right]
	\leqslant 0\, .
\end{align}

Combining the differential inequalities \eqref{dif-ineq2} and \eqref{dif-ineq3} for $\Norm{\varphi}^2$ and $|\psi|^2$, we get
\begin{align*}
	\frac12\frac{d}{dt}
	\bigg(
	\Norm{\varphi}^2&+(\nu\lambda_1)^2|\psi|^2
	\bigg)
	+\frac{\kappa(\nu\lambda_1)^2}{2}\Norm{\psi}^2
	+\frac{\nu}{8}|A_0\varphi|^2
	\\
	&+\Norm{\varphi}^2
	\bigg[
	\frac12\mup
	-\frac{g^2}{\kappa(\nu\lambda_1)^2}
	-\frac{\lambda_1\nu}{4}(K_2\log K_2)
	-\frac{2c_L^2(\nu\lambda_1)^2}{\kappa\lambda_1^2}\rho^2
	-\frac{2(\nu\lambda_1)^2}{l^2\lambda_1^2\kappa}
	\bigg]\\
	&\leqslant 
	\frac12\mup
	\Norm{\gamma}_X^2\nu^2\lambda_1\, .
\end{align*}
Consequently, by \eqref{mu4} and the Poincar\'e inequality, 
\begin{align}\label{ineq_Aphi}
	\frac{d}{dt}&
	\bigg(
	\Norm{\varphi}^2+(\nu\lambda_1)^2|\psi|^2
	\bigg)
	+\kappa(\nu\lam_1)^2\Norm{\psi}^2
	+\frac{\nu}{4}|A_0\varphi|^2
	+{\kappa\lambda_1}\Norm{\varphi}^2
	\\ \notag
	&\leqslant
	\frac{d}{dt}
	\bigg(
	\Norm{\varphi}^2+(\nu\lambda_1)^2|\psi|^2
	\bigg)
	+
	{\kappa\lambda_1}
	\bigg(
	\Norm{\varphi}^2+(\nu\lambda_1)^2|\psi|^2
	\bigg)
	+\frac{\nu}{4}|A_0\varphi|^2
	\\ \notag
	&\leqslant
	\mup\Norm{\gamma}_X^2
	\nu^2\lambda_1\, .
\end{align}
Dropping the terms $\frac{\nu}{4}|A_0\varphi|^2$ in the second inequality, using the Gronwall inequality and the fact that $\Norm{w_j}$, $|\eta_j|$ are bounded, we obtain
\begin{align}\label{Lipbnd}
	\Norm{\varphi}^2+(\nu\lambda_1)^2|\psi|^2
	\leqslant 
	\frac{\mup}{\kappa}
	\NormX{\gamma}^2
	\nu^2\, .
\end{align}
\subsubsection{Bound for $\intave  |A_0\varphi|^2$ and $\intave \Norm{\psi}^2$ by $\NormX{\gamma}^2$}
The inequality \eqref{ineq_Aphi} implies that
\begin{align*}
	\frac{d}{dt}
	\bigg(
	\Norm{\varphi}^2+(\nu\lambda_1)^2|\psi|^2
	\bigg)
	+\frac{\nu}{4}|A_0\varphi|^2
	+\kappa(\nu\lam_1)^2\Norm{\psi}^2
	\leqslant
	\mup\Norm{\gamma}_X^2
	\nu^2\lambda_1\, .
\end{align*}
Integrating from $t$ to $t+T$, $T=\tunit$, and using the bound \eqref{Lipbnd}, we have
\begin{align}\label{avebnds}
	\frac{\nu}{4}\intave  |A_0\varphi(\tau)|^2\,d\tau
	+\kappa(\nu\lam_1)\intave \Norm{\psi(\tau)}^2\,d\tau
	\leqslant
	\mup\left(\lambda_1T
	+	\frac{1}{\kappa}\right)
	\NormX{\gamma}^2
	\nu^2\,.
\end{align}

\subsubsection{Bounds for $\Norm{\psi}^2$ and $\intave |A_1\psi|^2$ by $\NormX{\gamma}^2$ }
Taking the $L^2$ inner product of \eqref{eq4-ch3-typeI} with $A_1\psi$, we have
\begin{align}\label{psiH1}
	\frac12\dt   \Norm{\psi}^2 
	+\kappa\Abs{A_1\psi}^2
	+b_1(w_2,\psi,A_1\psi)
	+b_1(\varphi,\eta_1,A_1\psi) 
	=\frac1l(\varphi\cdot \eb_2,A_1\psi),.
\end{align}
Integrating by parts, we have
\begin{align}\label{psiH1a}
	|b_1(w_2,\psi,A_1\psi)|
	&\leqslant
	\Norm{w_2}\cdot \Norm{\nabla\psi}_{L^4}^2
	\quad \text{(H\"{o}lder)}
	\\ \notag
	&\leqslant
	c_L \Norm{w_2}\cdot\Norm{\psi}\cdot|A_1\psi|
	\quad \text{(Ladyzhenskaya)}
	\\ \notag
	&\leqslant
	\frac{\kappa}{8}|A_1\psi|^2
	+\frac{2c_L^2}{\kappa}\Norm{w_2}^2\Norm{\psi}^2\,.
\end{align}
Similarly, 
\begin{align}\label{psiH1b}
	|b_1(\varphi,\eta_1,A_1\psi)|
	&\leqslant
	\Norm{\varphi}\cdot \Norm{\nabla\eta_1}_{L^4}\Norm{\nabla\psi}_{L^4}
	\quad \text{(H\"{o}lder)}
	\\ \notag
	&\leqslant 
	c_L \Norm{\varphi}\cdot\Norm{\eta_1}^{1/2}|A_1\eta_1|^{1/2} \Norm{\psi}^{1/2}|A_1\psi|^{1/2}
	\quad \text{(Ladyzhenskaya)}
	\\ \notag
	&\leqslant
	\frac{1}{2\nu} \Norm{\varphi}^2
	+\frac{c_L^2\nu}{2}\Norm{\eta_1}\cdot |A_1\eta_1|\cdot  \Norm{\psi}\cdot |A_1\psi|
	\quad \text{(Young)}
	\\ \notag
	&\leqslant
	\frac{1}{2\nu} \Norm{\varphi}^2
	+\frac{\kappa}{8}|A_1\psi|^2
	+\frac{c_L^4\nu^2}{2\kappa}\Norm{\eta_1}^2  |A_1\eta_1|^2\Norm{\psi}^2\,.
\end{align}

By Cauchy-Schwarz and Young inequalities,
\begin{align}\label{psiH1c}
	\frac1l|(\varphi\cdot \eb_2,A_1\psi)|
	\leqslant
	\frac{\kappa}{4}|A_1\psi|^2
	+\frac{1}{\kappa l^2}|\varphi|^2\,.
\end{align}

Combining \eqref{psiH1}--\eqref{psiH1c}, we obtain
\begin{align}\label{A1psi}
	\frac12\dt   \Norm{\psi}^2 
	+\frac{\kappa}{2}\Abs{A_1\psi}^2
	\leqslant	
	\left(
	\frac{2c_L^2}{\kappa}\Norm{w_2}^2
	+
	\frac{c_L^4\nu^2}{2\kappa}\Norm{\eta_1}^2  |A_1\eta_1|^2
	\right)
	\Norm{\psi}^2
	+
	\frac{1}{2\nu} \Norm{\varphi}^2
	+\frac{1}{\kappa l^2}|\varphi|^2\,.
\end{align}
Let the function $g$ and $h$ in Lemma \ref{Gronwall1} be
\begin{align}\label{gandh}
	g:=2\left(
	\frac{2c_L^2}{\kappa}\Norm{w_2}^2
	+
	\frac{c_L^4\nu^2}{2\kappa}\Norm{\eta_1}^2  |A_1\eta_1|^2
	\right),\quad
	h:= \frac{1}{\nu} \Norm{\varphi}^2
	+\frac{2}{\kappa l^2}|\varphi|^2\,.
\end{align}
By the bounds \eqref{wH1nst2}, \eqref{etaH1bnd10} and \eqref{ave_A1eta}, we have
\begin{align}\label{gbnd}
	\intave g(s)\,ds
	\leqslant
	\frac{4c_L^2}{\kappa} \cdot 4C_1\rho^2\lam_1T
	+
	\frac{c_L^4\nu^2}{\kappa^2}
	C_3
	\left[{C_3}+\left(
	\frac{8c_L^2 C_1\lambda_1C_3}{\kappa}
	+\frac{8C_1}{l^2\kappa}
	\right)
	\rho^2T
	\right]=:a_1\,.
\end{align}
By \eqref{Lipbnd} and the Poincar\'e inequality, we have
\begin{align}\label{hbnd}
	\intave h(s)\,ds\leqslant
	T\left(
	\frac{1}{2\nu}+\frac{1}{\kappa l^2\lam_1}
	\right)
	\frac{\mup}{\kappa}\NormX{\gamma}^2\nu^2=:K_{11}\NormX{\gamma}^2=:a_2
	\,.
\end{align}
By \eqref{avebnds}, 
\begin{align}\label{avepsi}
	\intave \Norm{\psi(\tau)}^2\,d\tau
	\leqslant
	\frac{\mup}{\kappa(\nu\lam_1)}
	\left(\lambda_1T
	+	\frac{1}{\kappa}\right)
	\NormX{\gamma}^2
	\nu^2=:K_{12}\NormX{\gamma}^2=:a_3\,.
\end{align}

Dropping the term $\frac{\kappa}{2}\Abs{A_1\psi}^2$ in \eqref{A1psi}, applying Lemma \ref{Gronwall1} with \eqref{gbnd}, \eqref{hbnd} and \eqref{avepsi} we have
\begin{align*}
	\sup_{t\in\R}\Norm{\psi(t)}^2
	\leqslant
	e^{a_1}
	\left(
	K_{11}
	+
	\frac{K_{12}}{T}
	\right)\NormX{\gamma}^2=:K_{13}\NormX{\gamma}^2\,.
\end{align*}
Now, by integrating \eqref{A1psi} from $t$ to $t+T$ and using \eqref{gbnd} and \eqref{hbnd}, we get
\begin{align*}
	\kappa\intave |A_1\psi(\tau)|^2\,d\tau
	\leqslant 
	(K_{13}+a_1K_{13}+K_{11})\NormX{\gamma}^2\,.
\end{align*}

\subsection{Stress-free BCs}\label{phi-bnd-sf-t1}

\subsubsection{Bounds for $|\psi|^2$, $|\varphi|^2$ and $\Norm{\varphi}^2$ by $\NormX{\gamma}^2$}
Taking the $L^2$ inner product  of \eqref{eq3-ch3-typeI}--\eqref{eq4-ch3-typeI} with $\varphi$ and $\psi$ respectively and taking the $L^2$ inner product of \eqref{eq3-ch3-typeI} with $A_0\varphi$ we have
\begin{align}\label{phi-eq1}
	\epsilon_1\left(
	\frac12\dt|\varphi|^2
	+\nu\Norm{\varphi}^2
	+b_0(\varphi,w_1,\varphi)
	\right)
	&=\epsilon_1\left(
	g(\psi \eb_2,\varphi)
	-\mup(\Iht \varphi-\gamma,\varphi)
	\right)\\
	\label{phi-eq2}
	\frac12\dt\Norm{\varphi}^2
	+\nu|A_0\varphi|^2
	+b_0(w_2,\varphi,A_0\varphi)
	&+b_0(\varphi,w_1,A_0\varphi)\\\notag
	&=g(\psi \eb_2,A_0\varphi)
	-\mup(\Iht \varphi-\gamma,A_0\varphi)\\
	\label{psi-eq1}
	\epsilon_2
	\left(
	\frac12\dt|\psi|^2+\kappa\Norm{\psi}^2+b_1(\varphi,\eta_1,\psi)
	\right)
	&=\epsilon_2\left(
	\frac1l(\varphi\cdot \eb_2,\psi)
	\right),
\end{align}
where, as in \eqref{weqn1}, \eqref{weqn2}, $\epsilon_1=|\Omega|^{-1}$, $\epsilon_2=(\nu\lam_1)^2$.

For the linear terms, as in \eqref{wineq1}--\eqref{wineq5} we have
\begin{align}\label{phi_ineq1}
	-{\mup}{\epsilon_1}(\Iht \varphi-\gamma, \varphi) 
	\leqslant
	\frac14\mup\VNorm{\varphi}^2
	+\mup{\epsilon_1}\Abs{\gamma}^2
	-\frac34{\mup}{\epsilon_1}|\varphi|^2
	+\frac\nu8|A_0\varphi|^2\,,
\end{align}

\begin{align}\label{phi_ineq1a}
	-\mup(\Iht \varphi-\gamma, A_0\varphi) 
	\leqslant
	\frac14\mup\VNorm{\varphi}^2
	+\frac{\nu}{8}|A_0\varphi|^2
	+\mup\Norm{\gamma}^2
	-\frac34\mup\Norm{\varphi}^2\,,
\end{align}
\begin{align}\label{phi_ineq2}
	{\epsilon_1}|(g\psi \eb_2,\varphi)|
	\leqslant
	\frac{\kappa\epsilon_2}{8}\Norm{\psi}^2
	+\frac{2}{\kappa\epsilon_2}
	\frac{g^2\lambda_1^{-1}}{|\Omega|}\VNorm{\varphi}^2,
\end{align}
\begin{align}\label{phi_ineq3}
	|(g\psi \eb_2, A_0\varphi)|
	\leqslant
	\frac{\kappa \epsilon_2}{8}\Norm{\psi}^2
	+\frac{2g^2}{\kappa \epsilon_2}\VNorm{\varphi}^2,
\end{align}
\begin{align}\label{phi_ineq4}
	\frac{\epsilon_2}{l}|(\varphi\cdot \eb_2,\psi)|
	\leqslant
	\frac{\kappa\epsilon_2}{4}\Norm{\psi}^2
	+\frac{\tilde K_1^2\epsilon_2}{\kappa l^2}\VNorm{\varphi}^2\,.
\end{align}

For the nonlinear terms, we have
\begin{align}\label{phi-NL-ineq1}
	\epsilon_1|b_0(\varphi,w_1,\varphi)|
	&\leqslant 
	\epsilon_1\Norm{w_1}\cdot\Norm{\varphi}_{L^4}^2
	\quad\text{(H\"{o}lder)}\\ \notag
	&\leqslant 
	\epsilon_1c_L\Norm{w_1}\cdot
	|\varphi|\cdot\VNorm{\varphi}
	\quad\text{(Ladyzhenskaya)}\\ \notag
	&\leqslant
	\epsilon_1c_L\tilde C_1^{1/2}\rho|\varphi|\cdot\VNorm{\varphi}
	\quad \textrm{(by (\ref{wVbnd}))}
	\\ \notag
	&\leqslant 
	\epsilon_1c_L\tilde C_1^{1/2}\rho|\Omega|^{1/2}\VNorm{\varphi}^2
	\quad \textrm{(by \eqref{u-Poin})}
\end{align}
\begin{align}\label{psi-NL-ineq1}
	\epsilon_2|b_1(\varphi,\eta_1,\psi)|
	&\leqslant 
	\epsilon_2\Norm{\varphi}_{L^4}
	\Norm{\eta_1}
	\Norm{\psi}_{L^4}
	\quad\textrm{(H\"{o}lder)}
	\\ \notag
	&\leqslant
	\epsilon_2c_L
	|\varphi|^{1/2}\VNorm{\varphi}^{1/2}\Norm{\eta_1}|\psi|^{1/2}\Norm{\psi}^{1/2}
	\quad \text{(Ladyzhenskaya)}
	\\ \notag
	&\leqslant
	\epsilon_2c_L|\Omega|^{1/4}
	\VNorm{\varphi}
	\Norm{\eta_1}
	\lambda_1^{-1/4}
	\Norm{\psi}
	\quad \textrm{(by \eqref{u-Poin})}
	\\ \notag
	&\leqslant
	\epsilon_2c_L|\Omega|^{1/4}\lambda_1^{-1/4}\tilde C_4^{1/2}\rho\VNorm{\varphi}\Norm{\psi}
	\quad\text{(by (\ref{c5}))}
	\\ \notag
	&\leqslant 
	\frac{\epsilon_2\kappa}{4}\Norm{\psi}^2
	+\frac{\epsilon_2K_{13}}{\kappa}\VNorm{\varphi}^2
	\quad \text{(Young)}\quad K_{13}:=c_L^2|\Omega|^{1/2}\lambda_1^{-1/2}\tilde C_4\rho^2
\end{align}
\begin{align}\label{phi2-NL-ineq1}
	|b_0(\varphi,&w_1,A_0\varphi)|
	\leqslant 
	\Norm{\varphi}_{L^\infty}\Norm{w_1}|A_0\varphi|
	\quad\text{(H\"{o}lder)}
	\\ \notag
	&\leqslant 
	c_A|\varphi|^{1/2}\Norm{\varphi}_{H^2}^{1/2}
	\Norm{w_1}|A_0\varphi|
	\quad \text{(2D Agmon)}
	\\ \notag
	&\leqslant c_Ac_E
	|\varphi|^{1/2}\bigg(
	\frac{1}{|\Omega|^{1/2}}|\varphi|^{1/2}
	+|A_0\varphi|^{1/2}
	\bigg)
	\Norm{w_1}|A_0\varphi|
	\quad \text{(by (\ref{elliptic-reg}))} 
	\\ \notag
	&\leqslant
	\frac{c_Ac_E}{|\Omega|^{1/2}} 
	|\varphi| \cdot|A_0\varphi|\tilde C_1^{1/2}\rho
	+c_Ac_E|\varphi|^{1/2}|A_0\varphi|^{3/2}\tilde C_1^{1/2}\rho
	\quad \text{ (by (\ref{wVbnd}))}\\ \notag
	&\leqslant 
	\frac{\nu}{4}|A_0\varphi|^2+\frac{K_{14}}{\nu}|\varphi|^2
	\quad \text{(Young)}
\end{align}
where
\begin{align*}
	K_{14}=2\hat{K_1}^2+\frac{54}{\nu^3}\hat{K_2}^4,\quad
	\hat{K_1}=\frac{c_Ac_E\tilde C_1^{1/2}\rho}{|\Omega|^{1/2}},
	\quad
	\hat{K_2}=c_Ac_E\tilde C_1^{1/2}\rho;
\end{align*}

\begin{align}\label{phi2-NL-ineq2}
	|b_0(w_2,&\varphi,A_0\varphi)|
	\leqslant 
	\Norm{w_2}_{L^4}\Norm{\nabla\varphi}_{L^4}\Abs{A_0\varphi}
	\quad\text{(H\"{o}lder)}\\ \notag
	&\leqslant 
	c_L|w_2|^{1/2}\Norm{w_2}_{H^1}^{1/2}
	|\nabla\varphi|^{1/2}\Norm{\nabla\varphi}_{H^1}^{1/2}|A_0\varphi|
	\quad \text{(Ladyzhenskaya)}\\ \notag
	&\leqslant 
	c_L|\Omega|^{1/4}\Norm{w_2}_{V_0}
	\|\varphi\|^{1/2}\Norm{\varphi}_{H^2}^{1/2}|A_0\varphi|\\ \notag
	&\leqslant 
	c_Lc_E|\Omega|^{1/4}\tilde C_1^{1/2}\rho\Norm{\varphi}^{1/2}
	\bigg(
	\frac{1}{|\Omega|^{1/2}}|\varphi|^{1/2}+|A_0\varphi|^{1/2}
	\bigg)
	|A_0\varphi|
	\quad \text{(by (\ref{elliptic-reg}))} 
	\\ \notag
	&\leqslant 
	c_Lc_E|\Omega|^{1/4}\tilde C_1^{1/2}\rho\Norm{\varphi}^{1/2}
	\bigg(
	\frac{|\Omega|^{1/4}}{|\Omega|^{1/2}}\VNorm{\varphi}^{1/2}+|A_0\varphi|^{1/2}
	\bigg)
	|A_0\varphi|
	\\ \notag
	&\leqslant 
	\frac{\nu}{4}|A_0\varphi|^2+\frac{K_{15}}{\nu}\VNorm{\varphi}^2
	\quad \text{(Young)}
\end{align}
where
\begin{align*}
	K_{15}=2\hat{K_3}^2+\frac{54}{\nu^3}\hat{K_4}^4,\quad
	\hat{K_3}=c_Ec_L\tilde C_1^{1/2}\rho,\quad
	\hat{K_4}=c_Ec_L|\Omega|^{1/4}\tilde C_1^{1/2}\rho.
\end{align*}

Combining \eqref{phi-eq1}--\eqref{phi2-NL-ineq2}, we have
\begin{align*}
	&\ \ \ \ \frac12\dt(\epsilon_1|\varphi|^2+\Norm{\varphi}^2+\epsilon_2|\psi|^2)
	+\epsilon_1\nu\Norm{\varphi}^2
	+\nu|A_0\varphi|^2
	+\epsilon_2\kappa\Norm{\psi}^2\\ \notag
	&\leqslant
	\VNorm{\varphi}^2\bigg[
	\frac12\mup-\frac34\mup
	+\frac{2g^2\lambda_1^{-1}}{\kappa\epsilon_2|\Omega|}
	+\frac{2g^2}{\kappa\epsilon_2}
	+\frac{\tilde K_1^2\epsilon_2}{\kappa l^2}\\ \notag
	&\qquad \qquad\qquad\qquad\qquad\qquad
	+\epsilon_1c_L\tilde C_1^{1/2}\rho|\Omega|^{1/2}
	+\frac{\epsilon_2K_{13}}{\kappa}
	+\frac{K_{15}}{\nu}
	+\frac{K_{14}|\Omega|}{\nu}
	\bigg]\\ \notag
	&\quad 
	+\mup\VNorm{\gamma}^2
	+\frac34\nu|A_0\varphi|^2
	+\frac34\kappa\epsilon_2\Norm{\psi}^2.
\end{align*}
It follows that
\begin{align*}
	\frac12\dt(\VNorm{\varphi}^2+\epsilon_2|\psi|^2)
	+\VNorm{\varphi}^2\bigg(
	\frac14\mup-K_{16}
	\bigg)
	+\frac14\epsilon_2\kappa\Norm{\psi}^2
	+\frac{\nu}{4}|A_0\varphi|^2
	\leqslant
	\mup\VNorm{\gamma}^2
\end{align*}
where
\begin{align}\label{k-16}
	K_{16}:=\frac{2g^2\lambda_1^{-1}}{\kappa\epsilon_2|\Omega|}
	+\frac{2g^2}{\kappa\epsilon_2}
	+\frac{\tilde K_1^2\epsilon_2}{\kappa l^2}
	+\epsilon_1c_L\tilde C_1^{1/2}\rho|\Omega|^{1/2}
	+\frac{\epsilon_2K_{13}}{\kappa}
	+\frac{K_{15}}{\nu}
	+\frac{K_{14}|\Omega|}{\nu}.
\end{align}
By \eqref{mu3-sf}, we have
\begin{align}\label{ineq_Aphi_sf}
	\dt(\VNorm{\varphi}^2+\epsilon_2|\psi|^2)
	+(\VNorm{\varphi}^2+\epsilon_2|\psi|^2)\frac{\kappa\lambda_1}{4}
	+\frac14\epsilon_2\kappa\Norm{\psi}^2
	+\frac{\nu}{2}|A_0\varphi|^2
	\leqslant
	2\mup\Norm{\gamma}_X^2\nu^2\lambda_1.
\end{align}
Dropping $\frac14\epsilon_2\kappa\Norm{\psi}^2+\frac{\nu}{2}|A_0\varphi|^2$ on the left and using the Gronwall inequality, we conclude that
\begin{align}\label{phi-bnd}
	\VNorm{\varphi}^2
	+\epsilon_2|\psi|^2
	\leqslant 
	\mu C_6\Norm{\gamma}_X^2,\quad 
	C_6:=\frac{8\lambda_1\nu^3}{\kappa},
\end{align}
and in particular, 
\begin{align}
	\VNorm{\varphi}^2
	\leqslant \mu C_6\Norm{\gamma}_X^2.
\end{align}

\subsubsection{Bound for $\intave \Norm{\psi}^2$ and $\intave  |A_0\varphi|^2$ by $\NormX{\gamma}^2$}
Using the inequality \eqref{ineq_Aphi_sf} and proceeding as in the no-slip case, we get 
\begin{align}\label{ave_bnds_sf}
	\epsilon_2\kappa\intave \Norm{\psi(\tau)}^2\,d\tau
	+{\nu}\intave  |A_0\varphi(\tau)|^2\,d\tau
	\leqslant
	(8\mup\nu^2\lambda_1T
	+4\mu C_6)
	\Norm{\gamma}_X^2\,.
\end{align}

\subsubsection{Bound for $\Norm{\psi}^2$ and $\intave  |A_1\psi|^2$ by $\NormX{\gamma}^2$}
Proceeding as in the no-slip case, we get \eqref{A1psi}:
\begin{align*}
	\frac12\dt   \Norm{\psi}^2 
	+\frac{\kappa}{2}\Abs{A_1\psi}^2
	\leqslant	
	\left(
	\frac{2c_L^2}{\kappa}\Norm{w_2}^2
	+
	\frac{c_L^4\nu^2}{2\kappa}\Norm{\eta_1}^2  |A_1\eta_1|^2
	\right)
	\Norm{\psi}^2
	+
	\frac{1}{2\nu} \Norm{\varphi}^2
	+\frac{1}{\kappa l^2}|\varphi|^2\,.
\end{align*}
Using the bounds \eqref{wVbnd}, \eqref{c5} and \eqref{Aeta_sf}, we have

\begin{align}\label{gbnd_sf}
	\intave g(s)\,ds
	\leqslant
	\frac{4c_L^2}{\kappa} \cdot \tilde{C}_1\rho^2T
	+
	\frac{c_L^4\nu^2}{\kappa^2}
	\tilde{C}_4\rho^2
	\left[
	\tilde C_4\rho^2
	+T\bigg({2\tilde K_3}\rho^4
	\tilde C_4\rho^2
	+\tilde K_4\rho^2
	\bigg)
	\right]=:a_1\,.
\end{align}
By \eqref{phi-bnd} and \eqref{u-Poin}, we have
\begin{align}\label{hbnd_sf}
	\intave h(s)\,ds\leqslant
	T\left(
	\frac{1}{\nu}+\frac{2|\Omega|}{\kappa l^2}
	\right)
	\mu{C}_6\NormX{\gamma}^2=:\tilde{K}_{11}\NormX{\gamma}^2
	\,.
\end{align}

Applying Lemma \ref{Gronwall1} with \eqref{gbnd_sf}, \eqref{hbnd_sf} and \eqref{ave_bnds_sf} yields
\begin{align*}
	\sup_{t\in\R}\Norm{\psi(t)}^2
	\leqslant
	e^{a_1}\left[
	\tilde{K}_{11}
	+
	\frac{1}{\kappa\epsilon_2T}(8\mup\nu^2\lambda_1T
	+4\mu C_6)
	\right]
	\NormX{\gamma}^2=:\tilde{K}_{12}\NormX{\gamma}^2\,.
\end{align*}

By integrating \eqref{A1psi} from $t$ to $t+T$ and using \eqref{gbnd_sf} and \eqref{hbnd_sf}, we get
\begin{align*}
	\kappa\intave |A_1\psi(\tau)|^2\,d\tau
	\leqslant 
	(\tilde{K}_{12}+a_1\tilde{K}_{12}+\tilde{K}_{11})\NormX{\gamma}^2\,.
\end{align*}

\section{Proof of Proposition \ref{prop-detForm-RB}}\label{pf2}
Let $\delta = w-u$ and $\xi = \eta-\theta$. Taking the difference of the RB system (\ref{eq-benardfn0}) and the auxiliary equations (\ref{detform-RB}), we have
\begin{subequations}
	\begin{align*}
		\frac{d\delta}{dt}+\nu A_0\delta+B_0(w,w)-B_0(u,u)
		&=\PL(g\xi \eb_2)-\mup (\Iht \delta),\\
		\frac{d\xi}{dt}+\kappa A_1\xi+B_1(w,\eta)-B_1(u,\theta)
		&=\frac{\delta\cdot \eb_2}{l}.
	\end{align*}
\end{subequations}
Applying the (essentially) same calculation in Section \ref{varphibnd},
we conclude that
\begin{align*}
	\Norm{\delta(t)}^2=|\xi(t)|^2=0,\quad\forall \, t\in\R,
\end{align*} 
which completes the proof.

\section{Appendix}
Let $\Tc(t;x)=\eta(t;x)+(1-\frac{x_2}{l})$ where $x=(x_1,x_2)\in \Omega$. Observe that for a given smooth enough $w$ with $\nabla\cdot w=0$,  $\Tc$ satisfies, on $[-\kt,T_{**})$,
\begin{align}\label{Tc}
	\frac{\partial \Tc}{\partial t} - \kappa \Delta\Tc +(w\cdot \nabla)\Tc=0,\\ 
	\Tc(-\kt;x_1,x_2)=1-\frac{x_2}{l}.
\end{align}
with boundary conditions
\begin{align*}
	\textrm{in the $x_2$-variable: }
	&\Tc=0\ \text{at $x_2=0$ and $x_2=l$,}
	\\
	\textrm{in the $x_1$-variable: }
	&\Tc\ \textrm{ is of periodic }L.
\end{align*}

Observe  that
$
0\leqslant \Tc(-\kt;x)\leqslant 1,
$
and thus
\begin{align*}
	\Tc_-(\kt;x)=0,\quad (\Tc-1)_+(\kt;x)=0,
\end{align*}
where we denote for any real number $M$,
$M_+=\max(M,0)$ and  $M_-=\max(-M,0).$

Note that $\Tct:=\Tc_-$ satisfies \eqref{Tc} a.e. and also the boundary conditions. 
The chain rule and integration by parts yield 
\begin{align*}
	\int_\Omega ((w\cdot \nabla)\Tct)\Tct\,dx
	=\sum_{i,j}\int_\Omega w_i(\partial_i\Tct_j)\Tct_j\,dx
	=\sum_{i,j}\int_\Omega w_i \partial_i\frac{(\Tct_j)^2}{2}\,dx
	=-\sum_{j}\int_\Omega (\nabla\cdot w)\frac{(\Tct_j)^2}{2}\, dx
	=0,
\end{align*}
where the boundary term vanishes due to the boundary conditions. 
Hence, multiplying \eqref{Tc} by $\Tc_-$ and integrating over $\Omega$, we obtain
\begin{align*}
	\frac12\dt |\Tc_-(t)|^2+\kappa|\nabla \Tc_-(t)|^2=0,
\end{align*}
which implies that 
\begin{align*}
	|\Tc_-(t)|^2\leqslant|\Tc_-(-\kt)|^2=0 \quad \textrm{for } t\in[-\kt,T_{**}). 
\end{align*}
It follows that $\Tc_-(t)=0$ and thus $\Tc(t)\geqslant 0$.

We now show that $\Tc\leqslant 1$. Observe that
\begin{align*}
	\frac{\partial}{\partial t}(\Tc-1)-\kappa\Delta(\Tc-1)+(w\cdot\nabla)(\Tc-1)=0.
\end{align*}
Proceeding similarly as above, we obtain, 
\begin{align*}
	\frac12\dt|(\Tc-1)_+|^2+\kappa|\nabla(\Tc-1)_+|^2=0,
\end{align*}
which implies that 
\begin{align*}
	|(\Tc-1)_+(t)|^2\leqslant|(\Tc-1)_+(-\kt)|^2=0 \quad \textrm{for } t\in[-\kt,T_{**}),
\end{align*}
and thus $\Tc(t)\leqslant 1$.

We conclude that 
\begin{align*}
	0\leqslant \Tc(t;x)\leqslant 1,\quad \textrm{a.e. } x\in \Omega,\ \  t\in[-\kt,T_{**}),
\end{align*}
which implies that
\begin{align*}
	|\eta(t,x)|\leqslant 1 + \sup_{x\in\Omega}|1-\frac{x_2}{l}|\leqslant 2,
\end{align*}
and thus
\begin{align}
	\Norm{\eta(t)}_{L^2(\Omega)}\leqslant 2|\Omega|,\quad 
	\forall\  t\in[-\kt,T_{**}).
\end{align}

\section{Acknowledgments}   
The work of Y. Cao was supported in part by National Science Foundation grant DMS-1418911, that of M.S. Jolly by  NSF grant DMS-1818754.
The work of E.S. Titi was supported
in part by the Einstein Visiting
Fellow Program, and by the John Simon Guggenheim Memorial Foundation.

\bibliographystyle{abbrv}
\bibliography{CJTref}

\end{document}